\documentclass[a4paper,12pt]{article}

\usepackage[bbgreekl]{mathbbol}
\usepackage{mathrsfs}
\usepackage{graphicx}
\usepackage{amsmath}
\usepackage{amsfonts}
\usepackage{indentfirst}
\usepackage{amsthm}
\usepackage{amssymb}
\usepackage{algorithm, algorithmic}
\usepackage{xcolor}
\usepackage{geometry}
\usepackage{url}
\usepackage{setspace}
\usepackage{verbatim}
\usepackage[normalem]{ulem}
\usepackage{multirow}
\usepackage{booktabs}
\usepackage[page,toc,titletoc,title]{appendix}
\usepackage{bm}
\usepackage{cases}
\usepackage{subfigure}
\usepackage{cancel}

\newtheorem{theorem}{Theorem}[section]
\newtheorem{lemma}{Lemma}[section]
\newtheorem{remark}{Remark}[section]

\newtheorem{definition}{Definition}[section]
\newtheorem{corollary}{Corollary}[section]

\newcommand{\R}{\mathbb{R}}
\newcommand{\C}{\mathbb{C}}

\newcommand{\N}{\mathbb{N}}
\newcommand{\LL}{\mathscr{L}}
\newcommand{\mv}{\bar{\varphi}} 
\newcommand{\ml}{\bar{\lambda}} 
\newcommand{\ev}{z} 
\newcommand{\evt}{\xi} 
\newcommand{\tv}{\psi} 
\newcommand{\rv}{f} 

\newcommand{\sad}{\bar{\it s}}

\newcommand{\sA}{\sigma_{A}}
\newcommand{\sB}{\sigma_{B}}

\newcommand{\dd}{~{\rm d}}
\newcommand{\for}{\quad{\rm for}~}

\newcommand{\tad}{\pmb{t}}
\newcommand{\Co}{C^1}    
\newcommand{\Cz}{C}

\newcommand{\U}{\mathscr{A}}  
\newcommand{\X}{X}  
\newcommand{\Y}{Y}
\newcommand{\HH}{\mathcal{H}}
\newcommand{\Nm}{\mathcal{H}_\perp}

\newcommand{\dE}{E_{\delta}}
\newcommand{\dF}{\mathscr{F}_{\delta}}
\newcommand{\mvp}{\mv_{\delta}}

\newcommand{\OA}{A}

\title{Stability of the Minimum Energy Path\thanks{
This work was supported by the National Key R\&D Program of China (No. 2020YFA0712900). HC's work was also supported by National Natural Science Foundation of China (No. NSFC11971066).
}
}

\author{
Xuanyu Liu\footnote{
    {\tt xyliu9535@mail.bnu.edu.cn}.
	School of Mathematical Sciences, Beijing Normal University, Beijing 100875, China.},
~ Huajie Chen\footnote{
    {\tt chen.huajie@bnu.edu.cn}.
	School of Mathematical Sciences, Beijing Normal University, Beijing 100875, China.}
~and~ Christoph Ortner\footnote{
    {\tt ortner@math.ubc.ca}.
	Department of Mathematics, University of British Columbia, 1984 Mathematics Road, Vancouver, BC, Canada V6T 1Z2.}
}

\date{}

\begin{document}
\maketitle
 
\begin{abstract}
The minimum energy path (MEP) is the most probable transition path that connects two equilibrium states of a potential energy landscape.
It has been widely used to study transition mechanisms as well as transition rates in the fields of chemistry, physics, and materials science. 
In this paper, we derive a novel result establishing the stability of MEPs under perturbations of the energy landscape. The result also represents a crucial step towards studying the convergence of various numerical approximations of MEPs, such as the nudged elastic band and string methods.
\end{abstract}


\section{Introduction}
\label{sec:introduction}
\setcounter{equation}{0}
The long term evolution of physical systems is often characterised by rare transitions between energy minima on a potential energy landscape. 
Within transition state theory, the leading-order contribution to the transition rate is the energy barrier between the minima.
One of the most popular methods for finding the energy barrier is to search for the minimum energy path (MEP) of the transition (see the review article \cite{2002surveyMEP}). This path is also interesting in its own right in that it provides modellers with insights into transition mechanisms. 
The MEP can be viewed as the most probable path of the transition between the minima.
The energy barrier (at the saddle) along the MEP path can then be used to calculate the transition rate by using approximations such as harmonic transition state theory \cite{berglund2013kramers,hanggi1990reaction,vineyard1957frequency}.
The most widely used techniques for finding the MEP are the nudged elastic band (NEB) method \cite{2000climbingNEB,jonsson1998NEB} and the string method \cite{2002string,2007string,2013climbingstring}.
They both iteratively evolve a discretised path of images in projected steepest descent directions, while keeping a smooth distribution of the images along the path
(see e.g. \cite{branduardi2013string,gardner2019energetics,NEBApp1,NEBApp3}).

To understand the rationale of modeling, the efficiency and the accuracy of numerical algorithms from a theoretical point of view, the stability of MEPs plays an important role. 
In particular, it is usually required that a small perturbation (or a good approximation) of the energy landscape does not lead to a large change in the MEP.
In spite of the importance of the stability result, there was very limited work on this aspect.
In \cite{cameron2011string}, the authors investigated the MEP from a dynamical systems point of view and studied the evolution of the path by the string method, establishing that the limiting curve is indeed MEP under certain conditions. 
In \cite{2018LuskinStabString}, the authors showed that the MEP is uniformly and asymptotically stable is the sense that any curve near the MEP can be arbitrarily close to it in the Hausdorff distance under the gradient decent dynamics with long enough evolution time.
Based on this stability, they also proved the convergence of the simplified and improved string method with respect to the time step and number of images on the curve. 
Neither of these stability results are strong enough to derive convergence \emph{rates}.

The purpose of this paper is to derive a ``strong" stability result for the MEP.
The novelty of our approach lies in that, we reformulate the MEP as the the root of an abstract operator, which is carefully constructed such that its linearisation becomes an isomorphism between appropriate function spaces.
With this approach, we can derive the stability that enables us to utilize powerful generic perturbation results such as the implicit function theorem to study the MEP.
In particular, it implies that a small perturbation of the energy landscape will lead to only a small deviation of the MEP, which shed lights on the understanding of both theoretical and numerical aspects for the MEP. Our analysis also lies the foundation for our approximation error analysis in \cite{dMEPconv}.

\vskip 0.1cm

\textbf{Outline.} 
The rest of this paper is organized as follows. 
In Section \ref{sec:results}, we present the main results of this paper, including the stability of the MEP, the intuition behind our construction, and an important application as corollary.
In Section \ref{sec:mep}, we present detailed proofs for the stability result, and provide a remark showing the necessity of the assumptions.
In Section \ref{sec:app}, we prove the corollary by using the stability result.
In Section \ref{sec:conclusion}, we give some conclusions.

\vskip 0.1cm

\textbf{Notations.}  
Let $X$ and $Y$ be Banach spaces with the norm $\|\cdot\|_X$ and $\|\cdot\|_Y$ respectively.
We will denote by $\LL(X,Y)$ the Banach space of all linear bounded operators from $X$ to $Y$ with the operator norm $\lVert\cdot\rVert_{\LL(X,Y)}$.
For a given functional $\mathscr{F}\in C^2(X)$ and $x\in X$, we will denote its first variation and second variation by $\delta\mathscr{F}(x)$ and $\delta^2\mathscr F(x)$, respectively. 
Let $\HH$ be a Hilbert space with the norm $|\cdot|$.
For $b>a$, we denote by $C\left([a,b];\HH\right)$ the space of continuous curves in $\HH$ with the norm $\lVert\varphi\rVert_{C([a,b];\HH)} := \sup_{\alpha\in[a,b]} |\varphi(\alpha)|$; 
and $C^1\left([a,b];\HH\right)$ the space of continuously differentiable curves with the norm $\lVert\varphi\rVert_{\Co([a,b];\HH)} := \lVert\varphi\rVert_{C([a,b];\HH)} + \lVert \varphi' \rVert_{C([a,b];\HH)}$. 
For a functional $E \in C^2(\HH)$ and $y\in\HH$, we will denote the gradient by $\nabla E(y)$, i.e., the Riesz-representer of the first variation $\delta E(y)$. 
And we will denote by $\nabla^2 E(y):\HH\rightarrow\HH$ the Hessian of $E$ (i.e. the Jacobian of $\nabla E$), with the inner product $(\nabla^2 E(y)x_1,x_2)$ being the Riesz representation of the second variation $\langle\delta^2 E(y)x_1,x_2\rangle$ for $x_1,x_2\in\HH$.
We will use $C$ to denote a generic positive constant that may change from one line to the next.
The dependencies of $C$ on model parameters (in our context, the energy landscape) will normally be clear from the context or stated explicitly.



\section{Main results}
\label{sec:results}
\setcounter{equation}{0}

\subsection{The Minimum energy path and its stability}
\label{sec:mep_s}

Let $E :\HH\rightarrow \R$ be a potential energy functional, where the configuration space $\HH$ is a Hilbert space with the inner product $(\cdot,\cdot)$ and the norm $|\cdot| := \sqrt{(\cdot,\cdot)}$. 
A configuration $y\in\HH$ could encode an atomic configuration, a crystalline structure, a phase field, and many other examples.
Throughout this paper, we will assume that $E \in C^3(\HH)$.
This regularity assumption is required because we need the Hessian $\nabla^2 E$ to be $C^1$ in our analysis.

Given an energy functional $E$, we call $y\in\HH$ a critical point if $\nabla E(y) = 0$. 
We call a critical point $y$ a strong local minimizer if the Hessian $\nabla^2 E(y)\in\LL(\HH)$ is positive definite, which means 
there is a positive constant $C>0$ such that
\begin{equation}
    \big( \nabla^2E(y) u,u\big) \geq C |u|^2\qquad\forall~u\in\HH.
\end{equation}
We call a critical point $y$ an index-1 saddle point if $\nabla^2 E(y)$ has exactly one negative eigenvalue while the rest spectrum are positive.
That is to say, there exist $\lambda_1<0$ and $v_1\in\HH$ such that
\begin{equation}
    \left\{\begin{aligned}
    &\nabla^2E(y) v_1 = \lambda_1 v_1, \quad{\rm and}
    \\[1ex]
    &\big( \nabla^2E(y) u,u\big) \geq C |u|^2\qquad\forall \quad u\in\HH, \quad (u,v_1) = 0. 
    \end{aligned}\right.
\end{equation}
For the sake of brevity, we will omit the qualifiers ``strong" and ``index-1" and simply say ``local minimizer" and ``saddle point".
We assume throughout that $E$ has at least two local minimizers on the energy landscape, denoted by $y^A_M\in\HH$ and $y^B_M\in\HH$ respectively.

A minimum energy path (MEP) is a curve $\varphi\in C^1\big([0, 1];\HH\big)$ connecting $y^A_M$ and $y^B_M$ whose tangent is everywhere parallel to the gradient except at the critical points. 
To give a rigorous definition of the MEP, we first introduce the projection operators $P_v~,P_v^{\perp}:\HH\rightarrow\HH$, given a $v\in\HH \setminus \{0\}$, 
\begin{equation*}
P_v y := \Big(y,\frac{v}{|v|}\Big) \frac{v}{|v|}
\qquad{\rm and}\qquad
P^\perp_{v}  := I - P_v ,
\end{equation*}
where $I$ is the identity operator.
We first define the admissible class for 
``regular" curves connecting the minimizers $y^A_M$ and $y^B_M$ as
\begin{align*}
\U := \Big\{ \varphi \in C^1\big( [0,1];\HH \big) :~  \varphi(0) = y^A_M,~\varphi(1) = y^B_M,~\varphi'(\alpha) \neq 0 ~~\forall~\alpha\in[0,1] \Big\} .
\end{align*}

Then, a MEP connecting the minimizers $y^A_M$ and $y^B_M$ is a solution of the following problem: Find $ \varphi \in \U $ such that
\begin{subequations} 
\label{mep}
	\begin{numcases}{}
	\label{mep1}\hspace{5pt}
	P^\perp_{\varphi'(\alpha)}\nabla E \big( \varphi(\alpha) \big) = {0} 	\hspace{6ex}~\forall~\alpha \in [0,1], \\[1ex]
	\label{mep2}\hspace{5pt}
	\Gamma(\varphi) = 0,
	\end{numcases}
\end{subequations}
where the operator $\Gamma: \Co([0,1];\HH)\rightarrow \Co([0,1];\R)$ is given by
\begin{equation*}
	\Gamma(\varphi)(\alpha) := \int_0^{\alpha} |\varphi'(s)|\dd s - \alpha \int_0^{1} |\varphi'(s)|\dd s\qquad{\rm for}~\alpha\in[0,1].
\end{equation*}
Note that \eqref{mep1} indicates that the gradient $\nabla E \big( \varphi(\alpha) \big)$ vanishes in the subspace perpendicular to the tangent $\varphi'(\alpha)$, which is well-defined since $\varphi\in\U$.
Further, \eqref{mep2} enforces the curve to be parameterized by normalized arc length, which removes the redundancy due to re-parameterization.

If $\mv$ is a solution of \eqref{mep}, since $y^A_M$ and $y^B_M$ are local minimizers, there exists an $\sad\in(0,1)$ with $y_S=\mv(\sad)\in\HH$ such that the energy $E(y_S)$ reaches the maximum along the MEP. 
This implies that $\nabla E(y_S)$ vanishes in the tangent direction $\mv'(\sad)$ and thus it is a critical point. 
Since the energy $E(y_S)$ is a maximum along the tangent $\mv'(\sad)$, we generically expect that the Hessian $\nabla^2E(y_S)$ has at least one negative eigenvalue.
For the sake of simplicity of the analysis, we will assume throughout this paper that  
\begin{flushleft}
\label{assump:index1saddle}
{\bf (A)}~
$y_M^A,~y_S,~y_M^B$ are the {\em only} critical points along the MEP $\mv$. Moreover, $y^A_M$ and $y^B_M$ are strong minimizers, while $y_S = \mv(\sad)$ is an {\em index-1 saddle}.
\end{flushleft}

Although {\bf (A)} is natural and will be satisfied by {\em many} (if not most) MEPs one encounters in practice, there are also cases where this assumption fails. 
For example, in \cite{MEP_MAP} examples are given where there is more than one {\em index-1 saddle} along an MEP.
Our theory can be generalized to these cases by adjusting the formulations, provided that all critical points along the MEP satisfy certain stability conditions. 
On the other hand, the strong stability assumption on the critical points cannot be readily weakened.

We observe by a direct calculation (see Lemma \ref{lemma:lambda}) that, if $\mv\in C^2\big( [0,1];\HH \big)$ solves \eqref{mep}, then $\mv'(0)$, $\mv'(\sad)$ and $\mv'(1)$ are eigenvectors of the Hessians $\nabla^2 E(y^A_M)$, $\nabla^2 E(y_S)$ and $\nabla^2 E(y^B_M)$, respectively.
This implies that the MEP has to go through the critical points in the direction of some eigenvector of the corresponding Hessian. 
The following assumption formalizes the requirement that
there is a unique optimal path to exit the energy minimizer. 

\begin{flushleft}
\label{assump:simplelowest}
{\bf (B)}~Let $\sA, \sB$ be the eigenvalues associated, respectively, with the eigenvectors $\mv'(0)$ and $\mv'(1)$. 
We assume that 
(i) they are lower bounds of the spectrum of $\nabla^2 E(y^A_M)$ and $\nabla^2 E(y^B_M)$, respectively; and 
(ii) they are simple and isolated eigenvalues. 
\end{flushleft}

Next, we rewrite the MEP equation \eqref{mep} in a form more convenient for our analysis. 
Let 
\begin{align}
\label{space:Y}
Y := \Big\{ \rv\in C\left( [0,1];\HH \right) :~ \rv(0) = \rv(1) = {0},~ f~\text{is differentiable at}~\alpha = 0, \sad, 1 \Big\}
\end{align}	
be the image space equipped with the following norm:
\begin{align}
\label{eq:norm:Y} 
\lVert f\rVert_{\Y} := \left\lVert \frac{ f(\alpha)}{\alpha(\alpha-1)}\right\rVert_{\Cz((0,1);\HH)} +  \left\lVert \frac{ f(\alpha) - f(\sad)}{\alpha-\sad}\right\rVert_{\Cz([0,\sad)\cup(\sad,1];\HH)}.
\end{align}
Then, we define 
$\mathscr{F} : \U \rightarrow \Y$ by 
\begin{align}
    \label{def:F}
    &\mathscr{F}(\varphi) (\alpha) :=
    P^\perp_{\varphi'(\alpha)} \nabla E\big(\varphi(\alpha)\big) - \frac{|\varphi'(\alpha)| - L(\varphi)}{L(\varphi)} P_{\varphi'(\alpha)} \nabla E\big(\varphi(\alpha)\big) 
    \\[1ex]
    &\hskip 0.5cm 
    + \frac{\alpha(\alpha-1)|\varphi'(\alpha)|}{\sad(\sad-1)L(\varphi)} P_{\varphi'(\alpha)} \nabla E\big(\varphi(\sad)\big) 
    + (\sA+\sB) \left( \Gamma(\varphi)(\alpha) - \frac{\alpha(\alpha-1)}{\sad(\sad-1)}\Gamma(\varphi)(\sad) \right)\frac{\varphi'(\alpha)}{\lvert \varphi'(\alpha) \rvert },
    \nonumber
\end{align}
for $\alpha\in[0,1]$, where $L : C^1([0,1];\HH) \rightarrow \R$ is the length function, 
\begin{equation*}
    L(\varphi) := \int_0^1 \lvert \varphi'(s) \rvert \dd s.
\end{equation*}
The weights $\alpha^{-1}(\alpha-1)^{-1},~(\alpha-\sad)^{-1}$ in the $\Y$-norm \eqref{eq:norm:Y} and in the construction of $\mathscr{F}$ are crucial for the stability analysis. 
We will provide an intuitive explanation in Section \ref{sec:stab_inter}, why $\mathscr{F}$ is formulated like this and how it affects the stability of MEP.
The following lemma shows that the range of $\mathscr{F}$ is in $\Y$, the proof of which is given in Section \ref{sec:mep}. 

\begin{lemma}\label{le:range}
$\mathscr{F}(\varphi) \in \Y$ for any $\varphi\in \U$.
\end{lemma}

Now we can rewrite the MEP equation as follows: 
Find $\varphi\in \U$ such that 
\begin{equation}
\label{pb:mep}
\mathscr{F}(\varphi) = 0.
\end{equation}
The first term in \eqref{def:F} is perpendicular to the tangent direction $\varphi'$, which is exactly the same as the left-hand side of \eqref{mep1}. 
The three remaining terms in \eqref{def:F} are parallel to the tangent $\varphi'$ and are designed to enforce the curve to be parameterized by normalized arc length. 
We can show some equivalence of \eqref{mep} and \eqref{pb:mep} in the following lemma, whose proof is given in Section \ref{sec:mep}.

\begin{lemma}
\label{le:MEP_equ}
Assume that $\mv\in \U$ solves \eqref{mep} and {\bf (A)} is satisfied. Then $\mv$ is a solution of \eqref{pb:mep}. 
Moreover, let $\varphi\in\U$ be a solution of \eqref{pb:mep} and let 
\begin{equation}
\tilde\varphi\left(\frac{\int_0^\alpha |\varphi'(s)|\dd s}{L(\varphi)}\right) := \varphi (\alpha) \qquad{\rm for}~\alpha\in[0,1],
\end{equation} 
then $\tilde\varphi \in\U$ and solves \eqref{mep}.
\end{lemma}

Since $E\in C^3$, we have that $\mathscr{F}$ is $C^1$ in a neighbourhood of $\U$.
We will denote the first variation of $\mathscr{F}$ by $\delta\mathscr{F}:X\rightarrow Y$ with
\begin{equation*}
    \X := \Big\{\tv\in \Co([0,1];\HH):~\tv(0)=\tv(1)=0 \Big\},
\end{equation*}
equipped with the norm $\|\tv\|_{\X}:=\|\tv\|_{\Co([0,1];\HH)}$.
The detailed expression for $\delta\mathscr{F}$ will be given in Section \ref{sec:mep}.
The following theorem is the main result of this paper, stating the continuity and stability of $\delta\mathscr{F}$.

\begin{theorem}
\label{theorem:Stability}
Assume that $\mv\in C^2\big( [0,1];\HH \big)$ solves \eqref{pb:mep} and {\bf (A)} is satisfied.
\begin{itemize}
\item[(i)]
Then there exists a $\delta_0 > 0$ such that for $\varphi_1,\varphi_2 \in B_{\delta_0}(\mv)\subset \U$, 
\begin{align*}
\lVert \mathscr{F}(\varphi_1) - \mathscr{F}(\varphi_2) \rVert_{\Y} 
&\leq C_0 \lVert \varphi_1-\varphi_2 \rVert_\X
\qquad \text{and}
\\[1ex]
\lVert \delta\mathscr{F}(\varphi_1) - \delta\mathscr{F}(\varphi_2) \rVert_{\LL(\X,\Y)}
&\leq C_1 \lVert \varphi_1-\varphi_2 \rVert_\X.
\end{align*}
where $C_0$ and $C_1$ are positive constants that depend only on $\delta_0,~E$ and $\mv$.
\item[(ii)]
If {\bf (B)} is also satisfied, then $\delta\mathscr{F}(\mv)$ is an isomorphism.
In particular, there exists a constant $\gamma>0$ depending only on $E$ and $\mv$ such that
\begin{equation}
\label{eq:stab}
\lVert \delta\mathscr{F}(\mv)^{-1} \rVert_{\LL(\Y,\X)} \leq \gamma.
\end{equation}
\end{itemize}
\end{theorem}

\begin{remark}
The condition {\bf (B)} is necessary for the stability of MEP since it enforces a unique direction along which the MEP can leave the minimizers, i.e. along the lowest-lying eigenvector $\mv'$.
If the corresponding eigenvalue at the minimizer is degenerate, then any perturbation within the eigenspace of this degenerated eigenvalue (which could be relatively large) may give rise to very small force.
In Section \ref{sec:finite}, we provide an example to demonstrate the loss of stability due to this degeneracy.
Moreover, if $\sA$ or $\sB$ are not the lowest eigenvalues, then a curve may favor leaving the minimizers along the lowest-lying eigenvector (which is not consistent with $\mv'$).
We provide in Section \ref{sec:finite} an example when $\delta\mathscr{F}(\mv)$ is not an isomorphism in such a situation.

We also remark on the, maybe surprising, fact that we did not require any assumption of stability of the potential energy orthogonal to the path (e.g., positivity of the hessian in directions normal to the MEP). While such a condition would be physically natural it is not required in our stability analysis since the stability is obtained via integrating a carefully formulated ODE {\em along} the path. 
\end{remark}

\subsection{Motivation of the stability analysis}
\label{sec:stab_inter}

In this subsection, we provide a motivation behind the construction of $\mathscr{F}$ and the weighted $\Y$-norm. 
Moreover, we will give an intuitive explanation on why the stability result in Theorem \ref{theorem:Stability} (ii) can hold.
Our stability result means that for a curve near the MEP, its deviation to the MEP can be controlled by its force under appropriate norms.
We equip the curve with the $\Co$-norm and the force with the $\Y$-norm \eqref{eq:norm:Y}, and formulate the stability of MEP as
\begin{equation*}
\|\varphi -\mv\|_{\X} \leq C \|\mathscr{F}(\varphi)\|_{\Y}
\end{equation*}
for any $\varphi\in\U$ in some neighborhood of $\mv$. 

The deviation in $C^1$-norm of a curve is determined by its oscillation and parameterization.
A key observation is that for a regular path $\varphi\in\U$, even if the force on the path has a small $C$-norm, it is still possible that the curve has strong oscillations near the critical points, $y_M^A,~y_S,~y_M^B$ (see Figure \ref{fig:oscillations} for a schematic plot).
However, replacing the $C$-norm with a suitably weighted variant, namely the $\Y-$norm \eqref{eq:norm:Y}, prevents such oscillations.

\begin{figure}
    \centering
    \includegraphics[width=8cm]{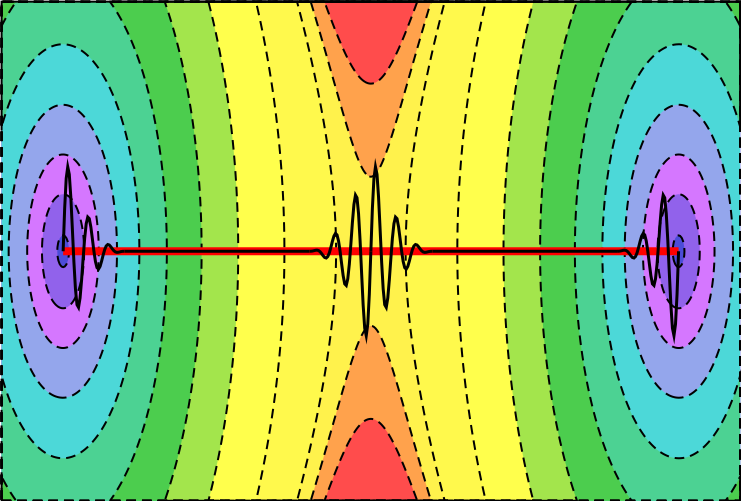}
    \caption{ MEP (red line) and a curve with oscillations near the critical points (solid black curve). The force in the $C$-norm, $\|\mathscr{F}(\varphi)\|_{C([0,1];\HH)}$, is small despite of large deviation in $C^1$-norm, $\|\varphi-\mv\|_{C^1([0,1];\HH)}$.}
    \label{fig:oscillations}
\end{figure}

To see this, we first consider the deviation at the minimizers at $\alpha=0,1$. We obtain by a direct computation that
\begin{align}
\label{eq:exp1}
\lim_{\alpha\rightarrow 0^+} \frac{\mathscr{F}(\varphi)(\alpha)}{\alpha(\alpha-1)} 
= P_{\varphi'(0)}^\perp \nabla^2 E(y^A_M)\varphi'(0) - \frac{|\varphi'(0)| - L(\varphi)}{L(\varphi)} P_{\varphi'(0)} \nabla^2 E(y^A_M)\varphi'(0).
\end{align}
Recall that the MEP goes through the minimizer $y^A_M$ in the direction of an eigenvector of the Hessian $\nabla^2 E(y^A_M)$. 
The first term of the right hand side in \eqref{eq:exp1} is perpendicular to tangent $\varphi'(0)$ and it can prevent the oscillation at the minimizer $y^A_M$. 
The second term of the right hand side in \eqref{eq:exp1} is parallel to the tangent and it can control the parameterization by the coefficient $|\varphi'(0)| - L(\varphi)$.

We turn to the deviation at the saddle at $\alpha=\sad$. We observe from \eqref{def:F} that
\begin{equation*}
\frac{\mathscr{F}(\varphi)(\sad)}{\sad(\sad-1)} = \frac{\nabla E\big(\varphi(\sad)\big)}{\sad(\sad-1)},
\end{equation*}
which restricts the distance between $\varphi(\sad)$ and the saddle $y_S$. 
Then, we obtain from a direct computation that
\begin{multline}
\label{eq:exp2}
\lim_{\alpha\rightarrow \sad} \frac{\mathscr{F}(\varphi)(\alpha)-\mathscr{F}(\varphi)(\sad)}{\alpha-\sad} 
=  P_{\varphi'(\sad)}^\perp \nabla^2 E\big(\varphi(\sad)\big)\varphi'(\sad) 
\\[1ex]
- \frac{|\varphi'(\sad)| - L(\varphi) }{L(\varphi)} P_{\varphi'(\sad)} \nabla^2 E\big(\varphi(\sad)\big)\varphi'(\sad) 
+ (\sA+\sB)\big(|\varphi'(\sad)| - L(\varphi) \big)  \frac{\varphi'(\sad)}{|\varphi'(\sad)|}.
\quad
\end{multline}
With the fact that $\varphi(\sad)$ is near the saddle $y_S$, the first term of the right hand side in \eqref{eq:exp2} could prevent the oscillation at the saddle. 
Besides, since the corresponding eigenvalue of $\mv'(\sad)$ is $\ml'(\sad) < 0$ (see Lemma \ref{lemma:lambda} (ii)) while $\sA,\sB>0$, the second and third terms of the right hand side in \eqref{eq:exp2} could control the parameterization at $\sad$.

For the deviation between the critical points, since the gradient never vanishes, the component of the force perpendicular to the tangent $P_{\mv'}^\perp\nabla E(\mv)$ could control the angle between the tangent and the gradient.
Thus it can prevent the oscillation. Moreover, the tangential force will naturally bound the parameterization.

\subsection{Applications}
\label{sec:applications}
The stability result in Theorem \ref{theorem:Stability} implies in particular that a small perturbation of the potential energy model gives rise to a small perturbation in the resulting MEP. In this section we state a concrete corollary explaining this statement in more detail. But one can expect even more general consequences for the stability of the MEP under perturbations of the entire MEP equations which in particular leads to novel analyses of the nudged elastic band and string methods that are explored in full detail in \cite{dMEPconv}. 

Fix an arbitrary $\epsilon > 0$. Let $\mathcal{U}_\epsilon := \bigcup_{\alpha \in [0, 1]} B_\epsilon(\bar\varphi(\alpha))$ where $B_\epsilon(y) := \{\tilde y\in\HH:~ |\tilde y - y|\leq \epsilon \}$.
Let $\HH_\delta$ be a sequence of subspaces of $\HH$ and let $\dE \in C^2(\HH_\delta; \R$) be an energy functional approximating $E|_{\HH_\delta}$. We assume {\em consistency} of this perturbed model, that is, 
 \begin{eqnarray}\label{eq:app1-5}
 \sup_{y \in \mathcal{U}_\epsilon} \inf_{y_\delta \in \mathcal{H}_\delta} |y_\delta - y|  +\|\dE - E\|_{C^2(\mathcal{U}_\epsilon\cap \mathcal{H}_\delta)} \rightarrow 0
 \quad{\rm as}~~\delta\rightarrow 0 .
 \end{eqnarray}
From assumption {\bf (A)}, if $\delta$ is sufficiently small (to ensure stability of the perturbed problem), then there exist two minimizers $y^A_{M,\delta},~y^B_{M,\delta}$ of $\dE$ satisfying
\begin{equation}
\label{eq:app1-4}
\begin{aligned}
|y^A_{M,\delta} - y^A_M| &\leq C \|\dE - E\|_{C^1(\mathcal{U}_\epsilon\cap \mathcal{H}_\delta)} + C \inf_{y_\delta\in\HH_\delta}|y_\delta-y^A_M|, 
\\[1ex]
|y^B_{M,\delta} - y^B_M| &\leq C \|\dE - E\|_{C^1(\mathcal{U}_\epsilon\cap \mathcal{H}_\delta)} + C \inf_{y_\delta\in\HH_\delta}|y_\delta-y^B_M|.
\end{aligned}
\end{equation}
Define
\begin{align*}
\U_\delta := \Big\{ \varphi \in C^1\big( [0,1];\HH_\delta \big) :~  \varphi(0) = y^A_{M,\delta},~\varphi(1) = y^B_{M,\delta},~\varphi'(\alpha) \neq 0 ~~\forall~\alpha\in[0,1] \Big\} .
\end{align*}
Then, a MEP in the perturbed model, connecting $y^A_{M,\delta}$ and $y^B_{M,\delta}$, is the solution of the following problem: Find $\varphi \in \U_\delta$ such that
\begin{equation} \label{mep-app1}
\left\{
	\begin{aligned}
	& P^\perp_{\varphi'(\alpha)}\nabla \dE \big( \varphi(\alpha) \big) = {0} 	\hspace{6ex}~\forall~\alpha \in [0,1], \\[1ex]
	& \Gamma(\varphi) = 0.
	\end{aligned}\right.
\end{equation}

\begin{corollary}
\label{coro:app1}
Let $\mv\in C^2\big([0,1];\HH\big)$ be the solution of \eqref{mep} and assume that {\bf (A)} and {\bf (B)} are satisfied.
If $\delta$ is sufficiently small, then \eqref{mep-app1} has a solution $\mvp\in\U_\delta$ such that
\begin{equation}
\label{eq:Corollar1}
\|\mv-\mvp\|_{C^1([0,1];\HH)} \leq C \|\dE - E\|_{C^1(\mathcal{U}_\epsilon\cap \mathcal{H}_\delta)} + C \inf_{\varphi_\delta\in\U_\delta}\|\varphi_\delta-\mv\|_{C^1([0,1];\HH)},
\end{equation}
where the constant $C$ depends only on $E$ and $\mv$.
\end{corollary}


\section{Proofs: Stability of the MEP}
\label{sec:mep}
\setcounter{equation}{0}

In this section, we will perform a careful analysis on the linearized operator $\delta\mathscr{F}$ and provide a proof for Theorem \ref{theorem:Stability}. 
We first provide some notations and preliminary results, then give the explicit formulation of $\delta\mathscr{F}$ and show its continuity, and finally prove that the linearized MEP operator $\delta\mathscr{F}(\mv)$ is an isomorphism.

\subsection{Some preliminaries}
\label{sec:notations}

We first define a function that measures the gradient of the MEP along the tangent,
\begin{equation}
\label{eq:ml}
\ml(\alpha) :=
\frac{1}{L(\mv)^2} \Big( \nabla E\big(\mv(\alpha)\big), \mv'(\alpha) \Big)
\qquad \for\alpha\in[0,1]
.
\end{equation}
It will be heavily used throughout our analysis.
We immediately see from \eqref{mep1}
and the fact $|\mv'|\equiv L(\mv)$
that
\begin{equation}
\label{eq:ml1}
\nabla E\big(\mv(\alpha)\big) = \ml(\alpha) \mv'(\alpha) \quad\for\alpha\in[0,1],
\end{equation}
which is the equation often used to define the MEP \cite{MEP_MAP,ramquet2000critical}. 
We state some properties of $\ml$ in the following lemma, from which we see that $\sA=\ml'(0)$ and $\sB=\ml'(1)$.

\begin{lemma}
\label{lemma:lambda}
Let $\mv\in C^2\big( [0,1];\HH \big)$ be the solution of \eqref{mep}.
If {\bf (A)} is satisfied, then
\begin{itemize}
\item[(i)]
$\big(\ml'(0), \mv'(0)\big)$, $\big(\ml'(\sad), \mv'(\sad)\big)$ and $\big(\ml'(1), \mv'(1)\big)$ are eigenpairs of the Hessians $\nabla^2 E(y^A_M)$, $\nabla^2 E(y_S)$ and $\nabla^2 E(y^B_M)$, respectively;
\item[(ii)]
$\ml'(0)>0,~\ml'(1)>0$ and $\ml'(\sad)<0$; and 
\item[(iii)]
there exist positive constants $\underline{c},~\bar{c}$ depending only on $\mv$, such that
\begin{equation}
\label{eq:mlequ}
    \underline{c} \leq \left\lvert \frac{\ml(\alpha)}{\alpha(\alpha-\sad)(\alpha-1)} \right\rvert \leq \bar{c} \qquad \forall~\alpha \in (0,\sad)\cup(\sad,1) .
\end{equation}
\end{itemize} 
\end{lemma}

\begin{proof}
Since the gradient vanishes at the critical points, we have $\ml(0)=\ml(\sad) = \ml(1) = 0$. 
Taking the derivative with respect to $\alpha$ on both sides of \eqref{eq:ml1}, we obtain
\begin{eqnarray}
\label{eq1}
\nabla^2 E\big(\mv(\alpha)\big) \mv'(\alpha) = \ml'(\alpha) \mv'(\alpha) + \ml(\alpha) \mv'' (\alpha).
\end{eqnarray}
Note that the second term on the right-hand side of \eqref{eq1} vanishes when $\alpha=0,\sad,1$, hence
\begin{eqnarray*}
\nabla^2 E\big(\mv(\alpha)\big) \mv' (\alpha) = \ml'(\alpha) \mv'(\alpha) \qquad{\rm for} ~\alpha=0,\sad,1. 
\end{eqnarray*}

We then obtain $\ml'(0)>0$ and $\ml'(1)>0$ from assumption ${\bf (A)}$ that the end points are strongly stable minimizers.
From the definition \eqref{eq:ml} and the MEP equation \eqref{mep}, we see that $\ml(\alpha) = 0$ if and only if $\mv(\alpha)$ is a critical point. 
Since $y^A_M,~y_S$ and $y^B_M$ are the only three critical points on the MEP (from assumption {\bf (A)}), we can deduce from $\ml'(0)>0,~\ml'(1)>0$ and the continuity of $\ml$ that $\ml(\alpha)>0$ for $\alpha\in(0,\sad)$ and $\ml(\alpha)<0$ for $\alpha\in(\sad,1)$.
This implies $\ml'(\sad) \leq 0$. Invoking the fact that $\ml'(\sad)$ is an eigenvalue and assumption {\bf (A)} that $y_S$ is an index-1 saddle, we see that the only possibility is $\ml'(\sad) < 0$.
	
Finally, (iii) is an immediate consequence of (ii) and the fact that $\ml$ has no roots other than $\alpha = 0, \sad, 1$.
\end{proof}

\begin{lemma}
\label{le:X-Y}
If $f\in \X\subset \Y$, then we have $\lVert f\rVert_{\Y}  \leq 3 \lVert f \rVert_{\X}$.
\end{lemma}

\begin{proof}
For any $f\in \X$, we have
\begin{align}
\label{eq:psad}
\left|\frac{f(\alpha)-f(\sad)}{\alpha-\sad}\right| 
&= \left| \frac{\int_{\sad}^{\alpha} f'(s) \dd s}{\alpha-\sad} \right| \leq \| f\|_{\X}
\qquad{\rm for}~\alpha\in[0,1]\backslash\{\sad\}.
\end{align}
Using $f(0)=0$, we have
\begin{align*}
\left|\frac{f(\alpha)}{\alpha(\alpha-1)}\right| &= \left| \frac{\int_0^{\alpha} f'(s) \dd s}{\alpha(\alpha-1)} \right| \leq 2 \| f\|_{\X} 
\qquad{\rm for}~\alpha\in\big(0,\frac{1}{2}\big].
\end{align*}
An analogous estimate holds for $\alpha\in [\frac{1}{2},1)$, which together with \eqref{eq:psad} yields the result.
\end{proof}

\vskip 0.2cm

We then rewrite \eqref{def:F} as
\begin{align}
\label{def:F2}
    \mathscr{F}(\varphi) = \nabla E(\varphi) - \frac{|\varphi'|}{L(\varphi)} P_{\varphi'} \tau(\varphi) + (\sA +\sB) \rho(\varphi) \frac{\varphi'}{|\varphi'|},
\end{align}
where $\tau : \U \rightarrow \X$ and $\rho: \U\rightarrow \Co([0,1];\R)$ are given by
\begin{align}\label{def:tau}
    \tau(\varphi)(\alpha) &:= \nabla E\big(\varphi(\alpha)\big) - \frac{\alpha(\alpha-1)}{\sad(\sad-1)} \nabla E\big(\varphi(\sad)\big) \qquad{\rm and}
     \\[1ex]
     \label{def:rho}
    \rho(\varphi)(\alpha) &:= \Gamma(\varphi)(\alpha) -  \frac{\alpha(\alpha-1)}{\sad(\sad-1)} \Gamma(\varphi)(\sad)
    \qquad{\rm for}~\alpha\in[0,1].
\end{align}

Now we are ready to prove Lemma \ref{le:range} and Lemma \ref{le:MEP_equ}, which state that the range of $\mathscr{F}$ is in $\Y$ and that the problem \eqref{mep} is equivalent to \eqref{pb:mep}, respectively.
For simplicity of notation, we suppress the dependence on $\alpha$ in the rest of this paper, whenever it is clear from the context.

\begin{proof}
[{\it Proof of Lemma \ref{le:range}}]
%
Since $\rho(\varphi)(0)=\rho(\varphi)(1) = \tau(\varphi)(0)=\tau(\varphi)(1) = 0$ for any $\varphi\in \U$, we see from \eqref{def:F2} that $\mathscr{F}(\varphi)(0) =\mathscr{F}(\varphi)(1) = 0$.

From the regularity of $E$, we have $\nabla E(\varphi),~\tau(\varphi) \in\X \subset \Y$.
Since $\varphi\in C^1$ and $\varphi'\neq 0$ for any $\varphi\in \U$, we have that $\frac{\varphi'}{|\varphi'|}\in C([0,1];\HH)$.
By using the continuity of $\varphi'$, the regularity of $E$, and the facts that $\tau(\varphi)(\sad)=0$ and $\rho(\varphi)(\sad)=0$, we have that the limit
\begin{equation*}
\lim_{\alpha \rightarrow \sad} \frac{\mathscr{F}(\varphi)(\alpha) - \mathscr{F}(\varphi)(\sad)}{\alpha-\sad}
\end{equation*}
exists. 
This implies that $\mathscr{F}(\varphi)$ is differentiable at $\sad$.
By an analogous argument, we see that $\mathscr{F}(\varphi)$ is also (one-sided) differentiable at $0,1$. 
Therefore, we have $\mathscr{F}(\varphi) \in\Y$.
\end{proof}

\begin{proof}
[{\it Proof of Lemma \ref{le:MEP_equ}}]
Let $\mv$ be the solutions of \eqref{mep}. Taking derivative of on both sides of \eqref{mep2} yields
\begin{equation*}
    |\mv'(\alpha)| - L(\mv) = 0 \qquad{\rm for}~\alpha\in[0,1].
\end{equation*}
Together with $\nabla E\big(\mv(\sad)\big) = 0$ and \eqref{mep}, we obtain $\mathscr{F}(\mv) = 0$, i.e., $\mv$ solves \eqref{pb:mep}.

If $\varphi$ is a solution of \eqref{pb:mep}, we see from the first term in \eqref{def:F} that $\varphi$ satisfies \eqref{mep1}. After re-parameterization by normalized arc length, \eqref{mep2} is satisfied while \eqref{mep1} still holds.
\end{proof}

\subsection{The linearized operator}
\label{sec:linearized}

For $\varphi\in \U$, the first variation $\delta\mathscr{F}(\varphi) : \X \rightarrow \Y$ can be obtained by a direct calculation
\begin{align}
\label{deltaF}
\delta \mathscr{F}(\varphi) \tv  &=  \nabla^2 E(\varphi)\tv  - \frac{|\varphi'|}{L(\varphi)} P_{\varphi'} \delta\tau(\varphi)\tv - \delta \left(\frac{|\varphi'|}{L(\varphi)} P_{\varphi'}\right)\tv \tau(\varphi) 
\nonumber \\[1ex]
& \hskip 3cm + (\sA+\sB) \left( \delta\rho(\varphi)\tv \frac{\varphi'}{|\varphi'|} + \rho(\varphi) \frac{ P_{\varphi'}^\perp \tv }{|\varphi'|}\right),
\end{align}
where 
\begin{align*}
    \delta \left( \frac{|\varphi'|}{L(\varphi)} P_{\varphi'}\right)\tv & = \frac{|\varphi'|}{L(\varphi)}\left( P_{\varphi'}^\perp \frac{\tv'(\varphi')^T}{|\varphi'|^2} + \frac{\varphi'(\tv')^T}{|\varphi'|^2}P_{\varphi'}^\perp \right)
    \\[1ex]
	&\hskip 1cm + \frac{1}{L(\varphi)} \left( \frac{\varphi'^T\tv'}{|\varphi'|} - \frac{|\varphi'|}{L(\varphi)} \int_0^1\frac{\varphi'(s)^T\tv'(s)}{|\varphi'(s)|}\dd s   \right) P_{\varphi'} ,
\end{align*}
and
\begin{align*}
	\big( \delta\tau(\varphi)\tv \big)(\alpha) &= \nabla^2 E\big(\varphi(\alpha)\big)\tv(\alpha) - \frac{\alpha(\alpha-1)}{\sad(\sad-1)} \nabla^2 E\big(\varphi(\sad)\big)\tv(\sad)\qquad{\rm for}~\alpha\in[0,1],
	\\[1ex]
	\delta\rho(\varphi)\tv &= \delta\Gamma(\varphi)\tv - \frac{\alpha(\alpha-1)}{\sad(\sad-1)} \big( \delta\Gamma(\varphi)\tv\big)(\sad),
	\\[1ex]
	\delta\Gamma(\varphi)\tv (\alpha) &= \int_0^{\alpha} \frac{\varphi'(s)^T\tv'(s)}{|\varphi'(s)|}\dd s - \alpha \int_0^1\frac{\varphi'(s)^T\tv'(s)}{|\varphi'(s)|}\dd s \qquad{\rm for}~\alpha\in[0,1].
\end{align*}
From the above expressions, we can easily show that $\mathscr{F}$ and $\delta\mathscr{F}$ are both Lipschitz in any bounded $C^1$-neighborhood of the MEP $\mv$.
This completes the proof of Theorem \ref{theorem:Stability} (i).

\subsection{$\delta\mathscr{F}(\mv)$ is an isomorphism}
\label{sec:LMEP}

The first variation \eqref{deltaF} at the MEP $\varphi=\mv$ can be simplified considerably by using \eqref{eq:ml1},
\begin{multline}
\label{eq:LMEPO}
\delta\mathscr{F}(\mv)\tv (\alpha) = 
P^\perp_{\mv'(\alpha)} \Big( \nabla^2 E\big(\mv(\alpha)\big)\tv(\alpha) - \ml(\alpha) \tv'(\alpha) \Big) + \frac{\alpha(\alpha-1)}{\sad(\sad-1)}P_{\mv'(\alpha)} \nabla^2 E\big(\mv(\sad)\big)\tv(\sad)
\\[1ex]
+ \left( (\sA+\sB)\left(\big(\delta\Gamma(\mv)\tv\big)(\alpha) - \frac{\alpha(\alpha-1)}{\sad(\sad-1)} \big(\delta\Gamma(\mv)\tv\big)(\sad) \right) - \ml(\alpha)\big(\delta\Gamma(\mv)\tv\big)'(\alpha) \right) \frac{\mv'(\alpha)}{\lvert \mv'(\alpha) \rvert} ,
~~
\end{multline}
for $\alpha\in[0,1]$, where $\ml$ is defined in \eqref{eq:ml}.

In order to show that the operator $\delta\mathscr{F}(\mv)$ is an isomorphism, i.e. Theorem \ref{theorem:Stability} (ii), we must establish the existence and stability of a solution for the following problem: 
Given $\rv\in\Y$, find $\tv\in\X$ such that
\begin{equation}
\label{eq:var}
\delta \mathscr{F}(\mv)\tv = \rv.
\end{equation}
The proof consists of three steps. 
First, we restrict the problem \eqref{eq:var} in the subspace perpendicular to the tangent direction $\mv'$ and derive an ODE system on a Hilbert space.
Secondly, we show that $\delta\mathscr{F}(\mv)$ is a bijection by proving that the ODE system has a unique solution. 
Finally, we derive the stability \eqref{eq:stab} by using the boundedness of $\delta\mathscr{F}(\mv)$ and the open mapping theorem.

As the linearized MEP operator \eqref{eq:LMEPO} has completely different behavior in the tangent direction $\mv'$ and in the subspace $\mv'^\perp$, we will split the problem into these two subspaces.
Let $\tad := \mv'/|\mv'|$.
We represent the trial function $\psi$ in \eqref{eq:var} by
\begin{equation}
\label{eq:tv}
\tv(\alpha) = \beta_0(\alpha) \tad (\alpha) + \tv_\perp(\alpha) \qquad{\rm for}~ \alpha\in[0,1] ,
\end{equation}
with $(\tv_\perp,\tad)=0$.

We first consider the problem \eqref{eq:var} in the subspace perpendicular to the tangent direction $\xi_0$. 
By substituting \eqref{eq:tv} into \eqref{eq:var} and applying the projection $P_{\mv'}^\perp$ to both sides, we can obtain by a direct calculation with \eqref{eq1} that
\begin{equation}
\label{pb:perp}
\ml(\alpha) \tv_\perp ' (\alpha) 
= \big( B(\alpha)+\ml(\alpha)D(\alpha) \big) \tv_\perp(\alpha) - P_{\mv'(\alpha)}^\perp \rv(\alpha)
\qquad{\rm for}~\alpha\in[0,1],
\end{equation}
where $B(\alpha) := P_{\mv'(\alpha)}^\perp \nabla^2E\big(\mv(\alpha)\big) P_{\mv'(\alpha)}^\perp$ 
and $D(\alpha):\HH\rightarrow\HH$ is given by
\begin{align*}
D(\alpha) y = - \big(y,\tad'(\alpha)\big)\tad(\alpha) - \big(y,\tad(\alpha)\big)\tad'(\alpha) .
\end{align*}
Note that the ODE \eqref{pb:perp} is independent of $\beta_0$.

Let
\begin{align*}
\HH^\perp(\alpha) := \big\{ y\in\HH:~\big(y,\mv'(\alpha)\big) = 0 \big\}
\qquad{\rm for}~ \alpha\in[0,1].
\end{align*}
Since $\mv\in C^2\big([0,1];\HH\big)$, there exists a Hilbert space $\Nm$ and a family of isometries $Q(\alpha):\Nm\rightarrow \HH^\perp(\alpha)$ for $\alpha\in[0,1]$ satisfying $Q \in C^1\big([0,1];\LL(\Nm,\HH)\big)$. 
The image of $Q(\alpha)$ is $\HH_\perp(\alpha)$, which is a subspace of $\HH$.
Let $\beta_\perp(\alpha) := Q(\alpha)^{-1}\tv_\perp(\alpha)$ for $\alpha\in[0,1]$. 
Then the ODE \eqref{pb:perp} can be rewritten as
\begin{equation}
\label{pb:beta}
\left\{
\begin{array}{ll}
\ml(\alpha) \beta_\perp ' (\alpha) = \OA(\alpha)\beta_\perp (\alpha)+ g(\alpha)\quad &\rm{for}~~ \alpha\in [0,1], \\[1ex]
\beta_\perp (0) = \beta_\perp (1) = {0},&
\end{array}
\right.
\end{equation}
where 
\begin{align}
\label{def:A}
A(\alpha) &:= Q(\alpha)^{-1}B(\alpha)Q(\alpha) - \ml(\alpha) \big( Q(\alpha)^{-1} D(\alpha)Q(\alpha) -   Q(\alpha)^{-1} Q'(\alpha)\big),
\\[1ex]
\label{def:g}
g(\alpha) &:= - Q(\alpha)^{-1}P_{\mv'(\alpha)}^\perp\rv(\alpha).
\end{align}
The advantage of \eqref{pb:beta} over \eqref{pb:perp} is that the unknown $\beta_\perp(\alpha)$ now lies in a space $\Nm$ that is independent of $\alpha$.

To write the problem in the tangent direction $\mv'$, we substitute \eqref{eq:tv} into \eqref{eq:var} and take the inner product with $\tad$ on both sides to derive the equation for $\beta_0$
\begin{multline}
\label{eq:beta0}
\ml(\alpha) \beta'_0(\alpha) 
= (\sA+\sB) \left( \beta_0(\alpha) - \frac{\alpha(\alpha-1)}{\sad(\sad-1)}\beta_0(\sad)\right) + \ml'(\sad)\frac{\alpha(\alpha-1)}{\sad(\sad-1)} \big( \tad(\alpha),\tad(\sad) \big)\beta_0(\sad) 
\\[1ex]
\qquad + \frac{\alpha(\alpha-1)}{\sad(\sad-1)} \big(\nabla^2E(\mv(\sad))\tv_\perp(\sad),\tad(\alpha)\big) - \big(\rv(\alpha),\tad(\alpha)\big) + G(\tv_\perp)(\alpha) 
\qquad{\rm for}~\alpha \in[0,1],
\end{multline}
where $G:\Co([0,1];\HH)\rightarrow \Cz([0,1];\R)$ is given by
\begin{align*}
& G(\tv_\perp)(\alpha) :=  (\sA+\sB)\left( \hat{G}(\tv_\perp)(\alpha) - \frac{\alpha(\alpha-1)}{\sad(\sad-1)}\hat{G}(\tv_\perp)(\sad)\right) - \ml(\alpha) \hat{G}(\tv_\perp)'(\alpha) 
\qquad{\rm with}
\\[1ex]
& \hat{G}(\tv_\perp)(\alpha) :=  \int_0^{\alpha} \big(\tv'_\perp(s), \tad(s)\big) \dd s - \alpha \int_0^1\big(\tv'_\perp(s), \tad(s)\big)\dd s.
\end{align*}
Since the tangential equation \eqref{eq:beta0} depends on $\tv_{\perp}$ (and hence $\beta_{\perp}$) but \eqref{pb:beta} does not depend on $\beta_0$, we will first focus on the problem \eqref{pb:beta} in the subspace $\mv'^\perp$.

\vskip 0.1cm

We first restrict the problem \eqref{pb:beta} to the interval $[\sad,1]$.
The interval $[0,\sad]$ can be analyzed with the same arguments. 
Thus, we consider the following problem:
\begin{equation}
\label{pb:beta-2}
\left\{
\begin{array}{ll}
\ml(\alpha) \beta_\perp ' (\alpha) = \OA(\alpha)\beta_\perp (\alpha) + g(\alpha)\quad &\rm{for}~~ \alpha\in [\sad,1], 
\\[1ex]
\beta_\perp (1) = {0}.&
\end{array}
\right.
\end{equation}
The main difficulty lies in that the prefactor $\ml(\alpha)$ vanishes at the critical points, $\ml(\sad)=\ml(1)=0$, which makes the standard arguments for ODE analysis fail. 
In our proof, we will first analyze the behavior of \eqref{pb:beta-2} near the critical points via semigroup theory, and then show the existence of $\beta_\perp$ over $[\sad,1]$ by a perturbation argument.

For completeness of the presentation, we briefly review the semigroup theory, starting with the definition of a sectorial operator.

\begin{definition}
\cite[Definition 2.0.1]{parasemigroup}
A linear operator $A$ on a Banach space $\mathcal{B}$ is a sectorial operator if there are constants $\omega\in\R,~\theta\in(\frac{\pi}{2},\pi),~M>0$ such that
\begin{equation}
\label{eq:sectorial}
\begin{aligned}
&\rho( A ) \supset S_{\omega,\theta}:= \{ \lambda \in \C : \lambda \neq \omega,~| {\rm arg}(\lambda - \omega) | < \theta \}
\qquad{\rm and}
\\[1ex]
& \|(\lambda I- A)^{-1}\|_{\LL(\mathcal{B})} \leq \frac{M}{|\lambda-\omega|}
\qquad\forall~\lambda\in S_{\omega,\theta}.
\end{aligned}
\end{equation}
\end{definition}

For $t>0$, \eqref{eq:sectorial} allows us to define a linear bounded operator $e^{tA}$ on $\mathcal{B}$, by means of the Dunford integral
\begin{equation}
\label{eq:semigroup1}
e^{t A} = \frac{1}{2\pi i} \oint_{\omega+\gamma_{r,\theta'}} e^{t\lambda} (\lambda I - A)^{-1} \dd\lambda,
\end{equation}
where $r>0$, $\theta' \in (\frac{\pi}{2},\theta)$ and $\gamma_{r,\theta'}$ is a contour $\{\lambda\in\C:~ |\text{arg}\lambda| = \theta',~|\lambda|\geq r\} \cup \{\lambda\in\C:~ |\text{arg}\lambda| \leq \theta',~|\lambda|= r\}$ that oriented counterclockwise. 
We also set
\begin{equation}
\label{eq:semigroup2}
e^{0 A} x = x \qquad \forall~x\in\mathcal{B}.
\end{equation}

\begin{definition}\cite[Definition 2.0.2]{parasemigroup}
Let $A\in\LL(\mathcal{B})$ be a sectorial operator. 
The family $\{e^{t A}:~t\geq 0\}$ defined by \eqref{eq:semigroup1} and \eqref{eq:semigroup2} is said to be the analytic semigroup generated by $A$ in $\mathcal{B}$.
\end{definition}

The following lemma presents some basic properties of the analytic semigroup.

\begin{lemma}
\cite[Proposition 2.0.1]{parasemigroup}
\label{le:semigroup}
Let $A\in\LL(\mathcal{B})$ be a sectorial operator satisfying \eqref{eq:sectorial} and $e^{tA}$ be the analytic semigroup generated by $A$, then
\begin{align*}
\|e^{t A}\|_{\LL(\mathcal{B})} \leq Ce^{\omega t}
\qquad{\rm and}\qquad
\frac{\dd}{\dd t}e^{tA} = A e^{tA}
\qquad \forall~t>0.
\end{align*}
\end{lemma}

Now we come back to analyze \eqref{pb:beta-2}.
Define
\begin{align}
\label{def:omega}
\nonumber
& \omega_S := \inf\{\lambda:~\lambda\in\sigma\big( \nabla^2E(y_S) \big),~\lambda>0\}
\qquad{\rm and}
\\[1ex]
& \omega_B := \inf\{\lambda:~\lambda\in\sigma\big( \nabla^2E(y^B_M) \big),~\lambda\neq\sB\}.
\end{align}
Let $A(\alpha)$ be the operator given in \eqref{def:A},  the following lemma shows that $-A(\alpha)$ at the critical points are sectorial.

\begin{lemma}
\label{le:A}
Let $\mv\in C^2\big([0,1];\HH\big)$ be the solution of \eqref{mep}, $\ml$ be given by \eqref{eq:ml} 
and $A$ be the operator-valued function defined in \eqref{def:A}. 
If assumptions {\bf (A)} and {\bf (B)} are satisfied, then $-A(\sad)$ and $-A(1)$ are sectorial operators, and $A$ belongs to  $C^1\big([\sad,1];\LL(\Nm)\big)$.
\end{lemma}

\begin{proof}
We see from {\bf (B)} that $B(\sad)|_{\HH^\perp(\sad)}$ and $B(1)|_{\HH^\perp(1)}$ are positive definite operators, and $\sigma\big( B(\sad)|_{\HH^\perp(\sad)} \big) \subset [\omega_S,+\infty)$, $\sigma\big( B(1)|_{\HH^\perp(1)} \big) \subset [\omega_B,+\infty)$.
Since $Q(\alpha):\HH^\perp(\alpha)\rightarrow\Nm$ is an isometry, we have from $\ml(\sad)= \ml(1)= 0$ that
\begin{align*}
& \sigma\big(A(\sad)\big) = \sigma\big(Q(\sad)^{-1}B(\sad)Q(\sad)\big) \subset [\omega_S,+\infty)
\qquad{\rm and}
\\[1ex]
& \sigma\big(A(1)\big) = \sigma\big(Q(1)^{-1}B(1)Q(1)\big) \subset [\omega_B,+\infty).
\end{align*}
This leads to
\begin{equation*}
\rho( -A(\sad) ) \supset S_{-\omega_S,\frac{3}{4}\pi}
\qquad{\rm and}\qquad
\rho( -A(1) ) \supset S_{-\omega_B,\frac{3}{4}\pi}.
\end{equation*}
Since $A(\sad)\in \LL(\Nm)$ there exists a constant $C>0$ depending on $A(\sad)$ such that
\begin{align}
\label{eq:sec1}
& \| \big(\lambda I + A(\sad) \big)^{-1} \|_{\LL(\Nm)} \leq \frac{1}{|\lambda|-\|A(\sad)\|_{\LL(\Nm)}}\leq \frac{C}{\lvert \lambda + \omega_S \rvert} ,
\end{align}
for $\lambda \in S(-\omega_S,\frac{3}{4}\pi) $ and $|\lambda| \geq 2\|A(\sad)\|_{\LL(\Nm)}$. 
For $\lambda \in S(-\omega_S,\frac{3}{4}\pi) $ and $|\lambda| \leq 2\|A(\sad)\|_{\LL(\Nm)}$, $\| \big(\lambda I + A(\sad) \big)^{-1} \|_{\LL(\Nm)}$ has an upper bound and thus \eqref{eq:sec1} holds for some constant $C$ depending on $A(\sad)$. 
This implies that $-A(\sad)$ is a sectorial operator. A similar argument yields that $-A(1)$ is also a sectorial operator.
	
From $E\in C^3(\HH)$ and $\mv\in C^2\big([0,1];\HH\big)$, we see that $B(\alpha)$ is differentiable in $[0,1]$. 
Combining this with the facts that $Q \in C^1\big([0,1];\LL(\Nm,\HH)\big)$ and $\ml(\sad)=\ml(1)=0$, we have that the operator-valued function $A \in C^1\big([\sad,1];\LL(\Nm)\big)$. 
\end{proof}

The following two lemmas give the behavior of \eqref{pb:beta-2} near the critical points.

\begin{lemma}
\label{le:beta1}
Let $\mv\in C^2\big([0,1];\HH\big)$ be the solution of \eqref{mep} and $\ml$ be given by \eqref{eq:ml}. Assume that {\bf (A)} and {\bf (B)} are satisfied. 
Let $\delta\in(0,\frac{1-\sad}{2})$ and $h(\alpha) \in C\big([\sad,\sad+\delta];\Nm\big)$ be differentiable at $\sad$.
Then there is a unique $\beta \in C^1\big([\sad,\sad+\delta];\Nm\big)$ solving
\begin{equation}
\label{pb:beta1}
\ml(\alpha) \beta'(\alpha) = \OA(\sad) \beta(\alpha) + h(\alpha)
\qquad{\rm for}~\alpha\in [\sad,\sad+\delta].
\end{equation}
Moreover, there exists a positive constant $C_{\rm S}$ depending only on $\ml$ and $A(\sad)$, such that
\begin{equation}
\label{est:beta1}
	\lVert  \beta' \rVert_{\Cz([\sad,\sad+\delta];\Nm)} \leq C_{\rm S} \left( \left\lVert \frac{h(\alpha)-h(\sad)}{\alpha-\sad} \right\rVert_{\Cz((\sad,\sad+\delta];\Nm)} + \| h(\sad)\|_{\Nm} \right).
\end{equation}
\end{lemma}

\begin{proof}
Using $\ml(\sad)=0$ and taking $\alpha=\sad$ in \eqref{pb:beta1} yields
\begin{equation}
\label{eq:4-3-1}
\beta(\sad) = -\OA(\sad)^{-1} h(\sad).
\end{equation}
We can now substitute 
\begin{align}
\label{eq:4-3-2}
\tilde{\beta}(\alpha) := \beta(\alpha) -  \beta(\sad)
\quad{\rm and}\quad
\tilde{h}(\alpha) := h(\alpha) -  h(\sad)
\qquad{\rm for}~\alpha\in[\sad,\sad+\delta]
\end{align}
and observe that $\beta$ solves \eqref{pb:beta1} if and only if $\tilde{\beta}$ solves
\begin{equation}
\label{eq:4-3-3}
    \left\{
	\begin{array}{ll}
	\ml(\alpha) \tilde{\beta}'(\alpha) = \OA(\sad) \tilde{\beta}(\alpha) + \tilde{h}(\alpha)\quad&{\rm for}~\alpha\in [\sad,\sad+\delta],
	\\[1ex]
	\tilde{\beta}(\sad) = 0. &
	\end{array}
	\right.
\end{equation}
Multiplying both sides of \eqref{eq:4-3-3} by $\ml(\alpha)^{-1}$, we obtain
\begin{equation}
\label{eq:4-3-4}
\tilde{\beta}'(\alpha) = \frac{\OA(\sad)}{\ml(\alpha)}\tilde{\beta}(\alpha) + \frac{\tilde{h}(\alpha)}{\ml(\alpha)}
\qquad{\rm for}~\alpha\in(\sad,\sad+\delta].
\end{equation}
Since $\ml<0$ in $(\sad,\sad+\delta]$ and $-A(\sad)$ is a sectorial operator, the solution $\tilde{\beta}$ of \eqref{eq:4-3-4} can be given by the method of variation of constants,
\begin{equation}
\label{eq:4-3-5}
\tilde{\beta}(\alpha) = e^{ \int_{\eta}^{\alpha}\frac{-1}{\ml(\tau)} \dd \tau \cdot (-\OA(\sad))}\tilde{\beta}(\eta) + \int_{\eta}^{\alpha} e^{ \int_{s}^{\alpha}\frac{-1}{\ml(\tau)} \dd \tau \cdot (-\OA(\sad))} \frac{\tilde{h}(s)}{\ml(s)}\dd s
\qquad{\rm for}~\alpha \in [\eta,\sad+\delta],
\end{equation} 
where $\eta\in(\sad,\sad+\delta)$ and $\tilde{\beta}(\eta)$ is the initial value.
Using Lemma \ref{lemma:lambda} (iii), Lemma \ref{le:semigroup} and Lemma \ref{le:A}, we have that, for $\eta <\alpha\leq\sad+\delta$,
\begin{equation}
\label{eq:4-3-6}
\left\| e^{ \int_{\eta}^{\alpha}\frac{-1}{\ml(\tau)} \dd \tau \cdot (-\OA(\sad))} \right\|_{\LL(\Nm)} \leq C e^{-\omega_S \int_{\eta}^{\alpha}\frac{-1}{\ml(\tau)} \dd \tau}
\leq C e^{-C \omega_S {\rm ln}\frac{\alpha-\sad}{\eta-\sad}}\leq C\left(\frac{\eta-\sad}{\alpha-\sad}\right)^{C\omega_S}<1 ,
\end{equation}
where $\omega_S$ is defined in \eqref{def:omega} and $C$ depends only on $E$ and $\mv$.
Then we can take the $\eta \rightarrow \sad^+$ limit of \eqref{eq:4-3-5}
\begin{equation}
\label{eq:4-3-7}
\tilde{\beta}(\alpha) = \int_{\sad}^{\alpha} e^{ \int_{s}^{\alpha}\frac{-1}{\ml(\tau)} \dd \tau \cdot (-\OA(\sad))} \frac{\tilde{h}(s)}{\ml(s)}\dd s.
\end{equation} 
By substituting \eqref{eq:4-3-7} into \eqref{eq:4-3-4}, using \eqref{eq:4-3-6} and Lemma \ref{lemma:lambda} (iii), we have
\begin{align*}
	\| \tilde{\beta}'(\alpha) \|_{\Nm} &= \left\| \frac{\OA(\sad)}{\ml(\alpha)} \int_{\sad}^{\alpha} e^{ \int_{s}^{\alpha}\frac{-1}{\ml(\tau)} \dd \tau \cdot (-\OA(\sad))} \frac{\tilde{h}(s)}{\ml(s)}\dd s + \frac{\tilde{h}(\alpha)}{\ml(\alpha)} \right\|_{\Nm}
	\\[1ex]
	&\hskip -1.5cm \leq C \left\| \frac{\tilde{h}(\alpha)}{\alpha-\sad} \right\|_{\Cz((\sad,\sad+\delta];\Nm)} (\alpha-\sad)^{-1} \int_{\sad}^{\alpha} 1 \dd s + C\left\| \frac{\tilde{h}(\alpha)}{\alpha-\sad} \right\|_{\Nm} \leq C \left\| \frac{\tilde{h}(\alpha)}{\alpha-\sad} \right\|_{\Cz((\sad,\sad+\delta];.\Nm)}.
\end{align*}
This together with \eqref{eq:4-3-1} and \eqref{eq:4-3-2} yields
\begin{align}\label{eq:4-3-8}
	\lVert  \beta' \rVert_{\Cz([\sad,\sad+\delta];\Nm)} &\leq C_{\rm S} \left( \left\lVert \frac{h(\alpha)-h(\sad)}{\alpha-\sad} \right\rVert_{\Cz((\sad,\sad+\delta];\Nm)} + \| h(\sad)\|_{\Nm} \right) ,
\end{align}
with a constant $C_{\rm S}$ depending only on $\ml$ and $A(\sad)$.
Moreover, since $h$ is differentiable at $\sad$, we see that $\tilde{\beta}(\alpha)$ in \eqref{eq:4-3-7} is also differentiable at $\sad$. 

Therefore, $\tilde\beta$ given by \eqref{eq:4-3-5} is in $\Co([\sad,\sad+\delta];\Nm)$, which readily implies that it indeed solves \eqref{eq:4-3-3}.
Together with \eqref{eq:4-3-1}, \eqref{eq:4-3-2} we therefore obtain a solution $\beta\in\Co([\sad,\sad+\delta];\Nm)$ to \eqref{pb:beta1}. 

To show the uniqueness of $\beta$ in \eqref{pb:beta1}, we only need to prove that there is only a trivial solution for the homogeneous equation
\begin{equation}
\label{eq:4-3-9}
\ml(\alpha) \beta'(\alpha) = \OA(\sad) \beta(\alpha) \quad{\rm for}~\alpha\in [\sad,\sad+\delta].
\end{equation}
By taking the inner product with $\beta(\alpha)$ on both sides, and using the facts that $Q(\alpha)$ is an isometry and that $B(\sad)$ is semi-positive definite, we have
\begin{align*}
\frac{\ml(\alpha)}{2}  \frac{\dd}{\dd\alpha} \|\beta(\alpha)\|_{\Nm}^2 
= \big(A(\sad)\beta(\alpha),\beta(\alpha)\big)_{\Nm} 
= \Big(B(\sad) Q(\alpha)^{-1}\beta(\alpha), Q(\alpha)^{-1}\beta(\alpha)\Big) \geq 0 
\end{align*}
for $\alpha\in[\sad,\sad+\delta]$.
This together with $\ml\leq 0$ in $[\sad,1]$ implies that $\|\beta\|_{\Nm}^2$ is decreasing in $[\sad,\sad+\delta]$. 
Combining this with $\beta(\sad)=0$ leads to $\beta\equiv 0$. 

Finally, the estimate \eqref{est:beta1} follows directly from \eqref{eq:4-3-8}, which completes the proof.
\end{proof}

\begin{lemma}
\label{le:beta2}
Let $\mv\in C^2\big([0,1];\HH\big)$ be the solution of \eqref{mep} and $\ml$ be given by \eqref{eq:ml}. 
Assume that {\bf (A)} and {\bf (B)} are satisfied. 
Let $v\in\Nm$, $\delta\in(0,\frac{1-\sad}{2})$, and $h(\alpha) \in C\big([1-\delta,1];\Nm\big)$ be differentiable at 1 and satisfy $h(1)=0$.
Then there is a unique $\beta \in C^1\big([1-\delta,1];\Nm\big)$ solving
\begin{equation}
\label{pb:beta2}
\left\{
\begin{array}{ll}
\ml(\alpha) \beta'(\alpha) = \OA(1) \beta(\alpha) + h(\alpha)
\quad &{\rm for}~\alpha\in [1-\delta,1],
\\[1ex]
\beta(1-\delta)=v,~\beta(1) = 0. &
\end{array}
\right.
\end{equation}
Moreover, there exists a positive constant $C_{\rm M}$ depending only on $\ml$ and $A(1)$, such that
\begin{align}
\label{est:beta2}
\lVert \beta' \rVert_{\Cz([1-\delta,1];\Nm)} &\leq C_{\rm M} \left( \left\lVert \frac{h(\alpha)}{\alpha-1} \right\rVert_{\Cz([1-\delta,1);\Nm)} + \delta^{-1} \|v\|_{\Nm}\right) .
\end{align}
\end{lemma}

\begin{proof}
Multiplying both sides of \eqref{pb:beta2} by $\ml(\alpha)^{-1}$, we have
\begin{equation}
\label{eq:4-4-1}
\beta'(\alpha) = \frac{A(1)}{{\ml}(\alpha)}\beta(\alpha) + \frac{h(\alpha)}{\ml(\alpha)}
\qquad{\rm for}~\alpha\in[1-\delta,1).
\end{equation}
Since $\ml<0$ in $(\sad,1)$ and $-A(1)$ is a sectorial operator, the solution ${\beta}$ of \eqref{eq:4-4-1} can be given by the method of variation of constants,
\begin{equation}
\label{eq:4-4-2}
\beta(\alpha) = e^{\int_{1-\delta}^\alpha \frac{1}{-\ml(s)} \dd s\cdot(-A(1))} v + \int_{1-\delta}^\alpha e^{ \int_s^{\alpha} \frac{1}{-\ml(\tau)} \dd \tau \cdot(-A(1)) } \frac{h(s)}{\ml(s)} \dd s
\qquad{\rm for}~\alpha \in [1-\delta,1).
\end{equation} 
Substituting \eqref{eq:4-4-2} into \eqref{eq:4-4-1} gives
\begin{align}
\label{eq:4-4-3}
\beta'(\alpha) = \frac{A(1)}{{\ml}(\alpha)}  e^{\int_{1-\delta}^\alpha \frac{1}{-\ml(s)} \dd s\cdot(-A(1))}v + \frac{A(1)}{{\ml}(\alpha)}\int_{1-\delta}^\alpha e^{ \int_s^{\alpha} \frac{1}{-\ml(\tau)} \dd \tau \cdot(-A(1)) } \frac{h(s)}{\ml(s)} \dd s  + \frac{h(\alpha)}{\ml(\alpha)}.
\end{align}
Using Lemma \ref{lemma:lambda} (iii), Lemma \ref{le:semigroup} and Lemma \ref{le:A}, we have that, for $1-\delta\leq s<\alpha<1$,
\begin{equation}
\label{eq:4-4-4}
\left\| e^{ \int_{s}^{\alpha}\frac{-1}{\ml(\tau)} \dd \tau \cdot (-\OA(1))} \right\|_{\LL(\Nm)} \leq C e^{-\omega_B \int_{s}^{\alpha}\frac{-1}{\ml(\tau)} \dd \tau}
\leq C e^{-\frac{\omega_B}{\ml'(1)} {\rm ln}\frac{1-s}{1-\alpha}}\leq C\left(\frac{1-\alpha}{1-s}\right)^{\sigma} ,
\end{equation}
where $\omega_B$ is defined in \eqref{def:omega} and $\sigma := \frac{\omega_B}{\ml'(1)}$. 
From assumption {\bf (B)} and Lemma \ref{lemma:lambda} (i), we see that $\sigma>1$.
Then we have from Lemma \ref{lemma:lambda} (iii), \eqref{eq:4-4-3} and \eqref{eq:4-4-4} that, for $\alpha\in[1-\delta,1)$,
\begin{align}
\label{eq:4-4-5}
\| \beta'(\alpha)\|_{\Nm}
&\leq C \frac{(1-\alpha)^{\sigma-1}}{\delta^{\sigma}} \|v\|_{\Nm} + C \int_{1-\delta}^{\alpha} \frac{(1-\alpha)^{\sigma-1}}{(1-s)^{\sigma}} \dd s \left\|\frac{h(s)}{s-1}\right\|_{C([1-\delta,1);\Nm)} + C\left\|\frac{h(\alpha)}{\alpha-1}\right\|_{\Nm}
\nonumber\\[1ex]
&\leq C_{\rm M} \left( \left\lVert\frac{h(s)}{s-1} \right\rVert_{\Cz([1-\delta,1);\Nm)}  + \delta^{-1} \| v\|_{\Nm} \right) ,
\end{align}	
where the constant $C_{\rm M}$ depends only on $\ml$ and $A(1)$. 
We see from assumption {\bf (B)} and Lemma \ref{le:A} that $\ml'(1)=\sB \in \rho\big( A(1)\big)$ and thus $I-\frac{A(1)}{\ml'(1)}$ is invertible.
Taking the limit $\alpha\rightarrow 1^-$ of \eqref{eq:4-4-3} gives
$\beta'(1) =\big(I-\frac{A(1)}{\ml'(1)}\big)^{-1}\frac{h'(1)}{\ml'(1)}$,
which implies the existence of $\beta'(1)$. 
Then taking the limit $\alpha\rightarrow 1^-$ of \eqref{pb:beta2} leads to $\beta(1) = 0$, and hence implies that there exists a $\beta\in C^1\big([1-\delta,1];\Nm\big)$ solving \eqref{pb:beta2}.

A similar argument as that in the proof of Lemma \ref{le:beta1} leads to the uniqueness of $\beta$ in \eqref{pb:beta2}. 
And the estimate \eqref{est:beta2} follows directly from \eqref{eq:4-4-5}, which completes the proof.
\end{proof}

Now we can show the existence and uniqueness of $\beta_\perp$ in \eqref{pb:beta-2}.

\begin{lemma}
\label{le:ext_beta}
For any $\rv\in\Y$, let $g$ be given by \eqref{def:g}.
If {\bf (A)} and {\bf (B)} are satisfied, then \eqref{pb:beta-2} has a unique solution $\beta_\perp \in \Co([\sad,1];\Nm)$. 
\end{lemma}

\begin{proof}
The proof for the existence and uniqueness of $\beta_\perp$ in \eqref{pb:beta-2} is divided into three parts.
First, we will prove the existence of $\beta_\perp$ in a neighborhood of $\sad$. 
Second, we will extend $\beta_\perp$ to the rest part of $[\sad,1]$ and verify the boundary condition that $\beta_\perp(1) = 0$. 
Finally, we show the uniqueness of the solution.

\vskip 0.1cm

1. {\it Existence near the saddle}.
Taking $\alpha=\sad$ in \eqref{pb:beta-2}, we have $\beta_\perp(\sad) = -A(\sad)^{-1}g(\sad)$.
Then we can take a sufficiently small $\delta_2\in(0,\frac{1-\sad}{2})$ such that 
\begin{align}
\label{eq:exist3}
2 C_{\rm S}\, \delta_2 \sup_{\alpha\in(\sad,1]}\left\|\frac{A(\alpha)-A(\sad)}{\alpha-\sad}\right\|_{\LL(\Nm)} \leq 1,	
\end{align}
where $C_{\rm S}$ is defined in \eqref{est:beta1}.

To show the existence of $\beta_\perp$, we construct a sequence $\{\beta_{\perp,k}\}_{k=0}^{\infty}$ in $[\sad,\sad+\delta_2]$ and show that it has a limit $\beta_\perp := \lim_{k\rightarrow\infty}\beta_{\perp,k}$, which solves \eqref{pb:beta-2}. 
Let $\beta_{\perp,0}(\alpha) \equiv \beta_\perp(\sad)$ for $\alpha\in[\sad,\sad+\delta_2]$. 
For $k\in\N_+$, let $\beta_{\perp,k+1}$ be the solution of 
\begin{align}
\label{eq:4-1-3}
\left\{
\begin{array}{ll}
\ml \beta'_{\perp,k+1} = A(\sad) \beta_{\perp,k+1} + \big(A - A(\sad) \big) \beta_{\perp,k}+ g \quad{\rm in}~[\sad,\sad+\delta_2], 
\\[1ex]
\beta_{\perp,k+1}(\sad) = \beta_\perp(\sad).
\end{array}
\right.
\end{align}
Lemma \ref{le:beta1} implies that the sequence $\{\beta_{\perp,k}\}_{k=0}^{\infty}$ is well defined and $\beta_{\perp,k} \in \Co([\sad,\sad+\delta_2];\Nm)$ for any $k\in\N$. 
We show that the sequence $\{\beta_{\perp,k}\}_{k=0}^{\infty}$ converges uniformly.
Define $e_{k} := \beta_{\perp,k} - \beta_{\perp,k-1}$ for $k\in\N_+$.
We see that $e_{k+1},~k\in\N_+$ is the solution of
\begin{align*}
	\left\{
	\begin{array}{ll}
	\ml e'_{k+1} = A(\sad) e_{k+1} + \big(A-A(\sad)\big) e_{k}  \quad{\rm in}~[\sad,\sad+\delta_2], \\[1ex]
	e_{k+1}(\sad) = 0.
	\end{array}
	\right.
\end{align*}
Applying Lemma \ref{lemma:lambda} (i), (iii), Lemma \ref{le:beta1} and \eqref{eq:exist3} yields
\begin{align}
\label{eq:4-1-1}
\nonumber
\|e'_{k+1}\|_{\Cz([\sad,\sad+\delta_2];\Nm)} 
&\leq C_{\rm S} \left\|\frac{\big(A(\alpha)-A(\sad)\big)e_{k}(\alpha)}{\alpha-\sad}\right\|_{\Cz((\sad,\sad+\delta_2];\Nm)}  
\\[1ex]
&\hskip -2cm 
\leq C_{\rm S}\, \delta_2 \sup_{\alpha\in(\sad,\sad+\delta_2]}\left\|\frac{A(\alpha)-A(\sad)}{\alpha-\sad}\right\|_{\LL(\Nm)}  \|e'_{k}\|_{\Cz([\sad,\sad+\delta_2];\Nm)} 
\leq \frac12 \|e'_{k}\|_{\Cz([\sad,\sad+\delta_2];\Nm)}
\end{align}
and
\begin{align}
\label{eq:4-1-2} 
\nonumber
\|e'_{1}\|_{\Cz([\sad,\sad+\delta_2];\Nm)}
& \leq C_{\rm S} \left\|\frac{g(\alpha) - g(\sad) +\big(A(\alpha)-A(\sad)\big)\beta_\perp(\sad) }{\alpha-\sad}\right\|_{\Cz((\sad,\sad+\delta_2];\Nm)} 
\\[1ex]
& \leq C_{\rm S}\left\|\frac{g(\alpha) - g(\sad) }{\alpha-\sad}\right\|_{\Cz((\sad,\sad+\delta_2];\Nm)} + C \|\beta_\perp(\sad)\|_{\Nm}.
\end{align}
Since $f\in\Y$, we see that $\|e'_{1}\|_{\Cz([\sad,\sad+\delta_2];\Nm)}$ is bounded. 
Then \eqref{eq:4-1-1} and \eqref{eq:4-1-2} imply that the sequence $\{\beta_{\perp,k}\}_{k=0}^{\infty}$ is is a Cauchy sequence in $\Co([\sad,\sad+\delta];\Nm)$. 
Let $\beta_\perp := \lim_{k\rightarrow\infty}\beta_{\perp,k}$, since the convergence is in $C^1$ we can pass to the limit in \eqref{eq:4-1-3}, which implies that $\beta_\perp$ solves \eqref{pb:beta-2} on $[\sad,\sad+\delta_2]$.
	
\vskip 0.2cm
	
2. {\it Existence away from the saddle}. %
Multiplying both sides of \eqref{pb:beta2} by $\ml^{-1}$, then according to the generalized Picard-Lindel\"{o}f theorem \cite[Theorem 5.2.4]{atkinson2005theoretical}, the existence of the solution ${\beta}_\perp$ in $[\sad+\delta_2,1)$ follows immediately from the continuity of $\ml,~A$, and $g$. 
Thus, there exists a $\beta_\perp\in\Co([\sad,1);\Nm)$ solving
\begin{equation}
\label{pb:beta_M}
\ml(\alpha)\beta'_\perp(\alpha) = A(\alpha) \beta_\perp(\alpha) + g(\alpha)
\qquad{\rm for}~\alpha\in[\sad,1).
\end{equation}
Take sufficiently small $\delta_3>0$ such that
\begin{equation*}
2 C_{\rm M}\, \delta_3 \sup_{\alpha\in[\sad,1)}\left\|\frac{A(\alpha)-A(1)}{\alpha-1}\right\|_{\LL(\Nm)} \leq 1,
\end{equation*}
where $C_{\rm M}$ are defined in \eqref{est:beta2}. 
To show the existence of $\displaystyle \lim_{\alpha\rightarrow 1^-}\beta'_\perp(\alpha)$, we will also construct a new sequence, still denoted by $\{\beta_{\perp,k}\}_{k=0}^{\infty}$. 
Let $\beta_{\perp,0}(\alpha) \equiv \beta_\perp(1-\delta_3)$ for $\alpha\in[1-\delta_3,1]$. 
For $k\in\N_+$, let $\beta_{\perp,k+1}$ be the solution of 
\begin{align}\nonumber
	\left\{
	\begin{array}{ll}
	\ml \beta'_{\perp,k+1} = A(1) \beta_{\perp,k+1} + \big(A - A(1) \big) \beta_{\perp,k} + g \quad{\rm in}~[1-\delta_3,1], \\[1ex]
	\beta_{\perp,k+1}(1-\delta_3) = \beta_\perp(1-\delta_3).
	\end{array}
	\right.
\end{align}
Lemma \ref{le:beta2} implies that the sequence $\{\beta_{\perp,k}\}_{k=0}^{\infty}$ is well defined and $\beta_{\perp,k} \in \Co([1-\delta_3,1];\Nm)$ for any $k\in\N$. 

We then show the uniform convergence of the sequence.
Let $e_{k} := \beta_{\perp,k} - \beta_{\perp,k-1}$ for $k\in\N_+$.
We see that $e_{k+1},~k\in\N_+$ is the solution of
\begin{align*}
	\left\{
	\begin{array}{ll}
	\ml e'_{k+1} = A(1) e_{k+1} + \big(A-A(1)\big) e_{k}  \quad{\rm in}~[1-\delta_3,1], \\[1ex]
	e_{k+1}(1-\delta_3) = 0.
	\end{array}
	\right.
\end{align*}
Applying Lemma \ref{le:beta2} and \eqref{eq:exist3} yields
\begin{align}
\label{eq:4-1-4}
	\|e'_{k+1}\|_{\Cz([1-\delta_3,1];\Nm)} &\leq C_{\rm M} \left\|\frac{\big(A(\alpha)-A(1)\big)e_{k}(\alpha)}{\alpha-1}\right\|_{\Cz([1-\delta_3,1);\Nm)} 
	\nonumber\\[1ex]
	&\hskip -2cm \leq C_{\rm M}\, \delta_3 \sup_{\alpha\in[1-\delta_3,1)}\left\|\frac{A(\alpha)-A(1)}{\alpha-1}\right\|_{\LL(\Nm)}  \|e'_{k}\|_{\Cz([1-\delta_3,1];\Nm)} \leq \frac12 \|e'_{k}\|_{\Cz([1-\delta_3,1];\Nm)}
\end{align}
and
\begin{align}
\label{eq:4-1-5}
\nonumber
\|e'_{1}\|_{\Cz([1-\delta_3,1];\Nm)}
& \leq C_{\rm M} \left\|\frac{g(\alpha) - g(1) +\big(A(\alpha)-A(1)\big)\beta_\perp(1-\delta_3) }{\alpha-1}\right\|_{\Cz([1-\delta_3,1);\Nm)}  
\\[1ex]
& \leq C_{\rm M}\left\|\frac{g(\alpha) - g(1) }{\alpha-1}\right\|_{\Cz([1-\delta_3,1);\Nm)} + C \|\beta_\perp(1-\delta_3)\|_{\Nm}, 
\end{align}
Since $f\in\Y$, we have that $\|e'_{1}\|_{\Cz([1-\delta_3,1];\Nm)}$ is bounded. 
Then \eqref{eq:4-1-4} and \eqref{eq:4-1-5} imply that the sequence $\{\beta_{\perp,k}\}_{k=0}^{\infty}$ is uniformly convergent. 
Let $k\rightarrow \infty$ in \eqref{eq:4-1-3}, we obtain that $\displaystyle \beta_{\perp} = \lim_{k\rightarrow\infty}\beta_{\perp,k}$ in $[1-\delta_3,1]$.
Then taking the $\alpha\rightarrow 1^-$ limit of \eqref{pb:beta_M} gives
\begin{equation*}
\beta_\perp(1) = -A(1)^{-1} g(1) = 0 ,
\end{equation*}
which completes the proof for the existence of the solution $\beta_\perp \in \Co([\sad,1];\Nm)$ of \eqref{pb:beta-2}.
    
\vskip 0.1cm
    
3. {\it Uniqueness}. 
It suffices to show that there is only a zero solution for the problem
\begin{equation}
\label{pb:uni}
	\left\{
	\begin{array}{ll}
	\ml(\alpha) \beta'_\perp(\alpha) = A(\alpha) {\beta}_\perp(\alpha)\quad{\rm for}~\alpha\in[\sad,1], \\[1ex]
	{\beta}_\perp(1) = {0}.
	\end{array}\right.
\end{equation}
Taking $\alpha=\sad$ in \eqref{pb:uni} yields $\beta_\perp(\sad) = 0$.
Then we take the inner product with $\beta_{\perp}(\alpha)$ on both sides of \eqref{pb:uni}.
Since $Q(\alpha)$ is an isometry, we have
\begin{multline}
\label{eq:4-1-6}
	\frac{\ml(\alpha)}{2}  \frac{\dd}{\dd\alpha} \|\beta_{\perp}(\alpha)\|_{\Nm}^2 
	= \big(  A(\alpha) {\beta}_\perp(\alpha), {\beta}_\perp(\alpha)\big)_{\Nm} = \Big(B(\alpha) Q(\alpha)^{-1}\beta_{\perp}(\alpha), Q(\alpha)^{-1}\beta_{\perp}(\alpha)\Big)
	\\[1ex]
	- \ml(\alpha) \left( \Big(Q(\alpha) D(\alpha)Q(\alpha)^{-1}\beta_\perp(\alpha),\beta_\perp(\alpha)\Big)_{\Nm}+ \Big(Q'(\alpha)Q(\alpha)^{-1}\beta_\perp(\alpha),\beta_\perp(\alpha)\Big)_{\Nm} \right) 
\end{multline}
for $\alpha\in[\sad,1]$. 
Since $B(\sad)$ and $B(1)$ are positive semi-definite and $\ml(\sad)=\ml(1)=0$, there exists a constant $\delta_4>0$ such that the right hand side of \eqref{eq:4-1-6} is non-negative for $\alpha\in[\sad,\sad+\delta_4]\cup[1-\delta_4,1]$. 
Using the fact that $\ml(\alpha) \leq 0$ for $\alpha\in[\sad,1]$, we obtain that $\|{\beta}_\perp(\alpha)\|_{\Nm}^2$ is decreasing on $[\sad,\sad+\delta_4]\cup[1-\delta_4,1]$. 
From $\beta_\perp(\sad)=0$, we have $\beta_\perp\equiv 0$ in $[\sad,\sad+\delta_4]$. 
Note that $\beta \equiv0$ is a solution of \eqref{pb:uni}, then using the generalized Picard-Lindel\"{o}f theorem \cite[Theorem 5.2.4]{atkinson2005theoretical} can yield that $\beta_\perp \equiv 0$ is the only solution of \eqref{pb:uni}. 
\end{proof}

\begin{corollary}
\label{coro:b}
Under the conditions of Lemma \ref{le:ext_beta}, \eqref{pb:beta} has a unique solution $\beta_\perp\in C^1\big([0,1];\Nm\big)$.
\end{corollary}

\begin{proof}
    The problem \eqref{pb:beta} can be divided into \eqref{pb:beta2} and the following problem
\begin{equation}\label{pb:beta3}
\left\{
\begin{array}{ll}
\ml(\alpha) \beta_\perp ' (\alpha) = A(\alpha)\beta_\perp (\alpha)+ g(\alpha)\quad &\rm{for}~~ \alpha\in [0,\sad], \\[1ex]
\beta_\perp (0) = 0.
\end{array}
\right.
\end{equation}
The same arguments as in Lemma \ref{le:ext_beta} can be used to obtain the existence and uniqueness of a solution ${\beta}_\perp \in \Co([0,\sad];\Nm)$ of \eqref{pb:beta3}. 
Then \eqref{pb:beta} has a unique solution $\beta_{\perp}$, which has continuous derivative on $[0,\sad]$ and $[\sad,1]$.
Taking $\alpha=\sad$ in \eqref{pb:beta}, we have $\beta_\perp(\sad) = -A(\sad)^{-1}g(\sad)$, which implies the continuity of $\beta_\perp$ at $\sad$.
By multiplying $\ml^{-1}(\alpha)$ on both sides of \eqref{pb:beta}, taking the $\alpha\rightarrow\sad$ limit and using the facts that $A(\sad)\beta_\perp(\sad) + g(\sad) = 0$ and $f$ is differentiable at $\sad$, we have
\begin{equation}
\label{eq:con_sad}
    \lim_{\alpha\rightarrow \sad^+}\beta'_\perp(\alpha)  = \frac{1}{\ml'(\sad)}\left( I - \frac{A(\sad)}{\ml'(\sad)}\right)^{-1} \lim_{\alpha\rightarrow\sad}\frac{g(\alpha)-g(\sad)}{\alpha-\sad} = \lim_{\alpha\rightarrow \sad^-}\beta'_\perp(\alpha).
\end{equation}
This yields $\beta_\perp\in\Co([0,1];\Nm)$.
\end{proof}

Now we are ready to complete the proof of Theorem \ref{theorem:Stability}.

\begin{proof}
[{\it Proof of Theorem \ref{theorem:Stability} (ii)}]
We first show that $\delta\mathscr{F}(\mv)$ is a bijection.
Let $\rv\in\Y$ and $g$ be given by \eqref{def:g}. From Corollary \ref{coro:b}, we obtain a unique solution $\beta_\perp\in C^1\big([0,1];\Nm\big)$ of \eqref{pb:beta} and let $\tv_\perp = Q^{-1}\beta_\perp$.

Next, we show the existence of a unique solution $\beta_0\in\Co([0,1];\R)$ of \eqref{eq:beta0}, i.e., the existence of $\tv$ along the tangent direction. 
By taking $\alpha=\sad$ in \eqref{eq:beta0}, we see that $\beta_0(\sad)$ can be uniquely determined by $f$ and $\beta_\perp$, more precisely, $\beta_0(\sad) = \ml'(\sad)^{-1}\Big( \big( f(\sad),\tad(\sad)\big) - G(\tv_\perp)(\sad)\Big)$. 
Then we can rewrite \eqref{eq:beta0} as
\begin{align}
\label{eq:beta0-2}
\ml(\alpha) \beta'_0(\alpha) = (\sA+\sB) \beta_0(\alpha) + g_0(\alpha) 
\qquad{\rm for}~\alpha \in[0,1],
\end{align}
where 
\begin{align*}
g_0(\alpha) &:= \ml'(\sad)\frac{\alpha(\alpha-1)}{\sad(\sad-1)} \big( \tad(\alpha),\tad(\sad)\big)\beta_0(\sad)  - (\sA+\sB)\frac{\alpha(\alpha-1)}{\sad(\sad-1)}\beta_0(\sad)
\\[1ex]
&\quad + \frac{\alpha(\alpha-1)}{\sad(\sad-1)} \big(\nabla^2E(\mv(\sad))\tv_\perp(\sad),\tad(\alpha)\big) - \big(\rv(\alpha),\tad(\alpha)\big) + G(\tv_\perp)(\alpha),
\quad{\rm for}~\alpha \in[0,1].
\end{align*}
Using Lemma \ref{le:X-Y}, we have 
\begin{align*}
    \|g_0\|_{\Y} & \leq C \|\beta_0(\sad)\|_{\Nm} + \|G(\beta_\perp)\|_{\Y} + C \|f\|_{\Y} \leq C \|f\|_{\Y} + \|G(\beta_\perp)\|_{\Co([0,1];\R)} 
    \\[1ex]
    & \leq C \|f\|_{\Y} + \|\beta_\perp\|_{\Cz([0,1];\Nm)}. 
\end{align*}
Since $\sA+\sB > \ml'(1) > 0$, an argument analogous to that in the proof of Lemma \ref{le:ext_beta} implies that there exists a unique $\beta_0\in\Co([0,1];\R)$ solving \eqref{eq:beta0}. 
With the representation \eqref{eq:tv}, we show that \eqref{eq:var} has a unique solution $\tv\in\X$ for any $f\in\Y$, which implies that $\delta\mathscr{F}(\mv)$ is a bijection.

For any $\tv\in\X$, we have from Lemma \ref{le:X-Y} and \eqref{eq:LMEPO} that
\begin{equation}\nonumber
    \|\delta\mathscr{F}(\mv)\tv\|_{\Y} \leq C \big( \|\tv\|_{\X} + \|\delta\Gamma(\mv)\tv\|_{\Co([0,1];\R)}\big) \leq C \|\tv\|_{\X}.
\end{equation}
Thus, $\delta\mathscr{F}(\mv)$ is a bounded linear bijection and according to the open mapping theorem (or, bounded inverse theorem) it is an isomorphism. 
This completes the proof of Theorem \ref{theorem:Stability}.
\end{proof}

\subsection{Remarks on assumption {\bf (B)}}
\label{sec:finite}

In Theorem \ref{theorem:Stability} (ii), the condition that $\sA$ and $\sB$ are simple and the lowest eigenvalues is necessary for the stability result.
We provide two examples in the following, showing that if $\sA$ is a degenerated eigenvalue or not the lowest eigenvalue, then the stability result in Theorem \ref{theorem:Stability} (ii) does not hold. 

For simplicity of the presentations, we consider a finite dimensional configuration space $\HH =\R^{N}$. 
We first give the isometry $Q(\alpha)$. 
Consider the eigenvalue problem 
\begin{equation}
\label{eq:ev}
\Big(P^\perp_{\mv'(\alpha)}\nabla^2 E\big(\mv(\alpha)\big) P^\perp_{\mv'(\alpha)}\Big) \evt_i(\alpha) = \ev_i(\alpha) \evt_i(\alpha)
\qquad {\rm for}~ i = 0,1,\cdots,N-1 
\end{equation}
with $\{z_i(\alpha)\}_{i=0}^{N-1}$ the eigenvalues and $\{\evt_i(\alpha)\}_{i=0}^{N-1}$ the corresponding eigenfunctions.
We order the functions $z_i(\alpha)$ according to their values at $\alpha=\sad$ and their regularity with respect to $\alpha$, such that $\ev_0(\sad) \leq \ev_1(\sad) \leq \cdots \leq \ev_{N -1}(\sad)$ and $z_i\in C^1([0,1];\R)$ for $i=0,1,\cdots,N-1$.
This can be done since $E\in C^3$ and  $P^\perp_{\mv'(\alpha)}\nabla^2 E\big(\mv(\alpha)\big) P^\perp_{\mv'(\alpha)}\in C^1\left([0,1];\R^{N\times N}\right)$, which together with the perturbation theory for eigenvalue problems \cite[II-Theorem 6.8]{kato2013perturbation} can lead to $z_i\in C^1([0,1];\R)$. 
Similarly, we have $\evt_i \in C^1\left([0,1];\R^{N}\right)$ for $i=0,1,\cdots, N-1$.
Define $Q(\alpha) : \HH^\perp(\alpha)\rightarrow \R^{N-1}$ by \begin{equation*}
	Q(\alpha) y:= \Big( \big(y,\evt_1(\alpha)\big),~\big(y,\evt_2(\alpha)\big),\cdots,\big(y,\evt_{N-1}(\alpha)\big) \Big)^T.
\end{equation*}
We see immediately that $Q(\alpha)$ is an isometry for $\alpha\in[0,1]$ and is differentiable at $0,\sad,1$.
We will now show that the assumption {\bf (B)} is necessary for the stability result.

Assume that there exists $\ev_j(0)=\sA=\ml'(0)$ with $1\leq j\leq N-1$ (which is contradicted to the assumption that $\sA$ is simple).
Then we have from $E\in C^3$ that there exists $\eta_0\in(0,\frac{\sad}{2})$ such that $\lVert \ev_j - \ml' \rVert_{\Cz([0,\eta_0];\R)} \leq 1$.
Let
\begin{equation}
\label{eq:rem1-2}
\beta_{n}(\alpha) = \left\{
\begin{array}{ll}
\left( \frac{\eta_0-\alpha}{\eta_0} \right)^{n} \ml(\alpha) &{\rm in} ~~[0,\eta_0]  \\[1ex]
0 &{\rm in} ~~(\eta_0,1]
\end{array} ,
\right.
\quad
\tv_n(\alpha) = \beta_n(\alpha)\evt_j(\alpha) 
\quad{\rm and}\quad
f_n = \delta\mathscr{F}(\mv)\tv_n .
\end{equation}
We have that $f_n(\alpha)=0 $ for $ \alpha\in [\eta_0,1]$. 
For $\alpha\in[0,\eta_0]$ we obtain, by using \eqref{pb:beta}, \eqref{eq:beta0} and \eqref{eq:ev}, that
\begin{align}
\label{eq:rem1-3}
\evt_j(\alpha)^T f_n(\alpha) &= \ev_j(\alpha)\beta_n(\alpha) - \ml(\alpha)\beta'_n(\alpha) 
\nonumber \\[1ex]
&= \ml(\alpha)\big(\ev_j(\alpha) -\ml'(\alpha)\big)
\Big( \frac{\eta_0-\alpha}{\eta_0} \Big)^n - \ml(\alpha)^2\frac{n( \eta_0-\alpha)^{n-1}}{\eta_0^{n}},
\nonumber \\[1ex]
\evt_i(\alpha)^T f_n(\alpha) &= -\ml(\alpha) \beta_n(\alpha) \evt_i(\alpha)^T\evt'_j(\alpha) \quad{\rm for}~ i = 1,\cdots,N-1,~i\neq j , 
\qquad{\rm and}
\nonumber \\[1ex]
\evt_0(\alpha)^T f_n(\alpha) & = (\sA+\sB) \left( \hat{G}_n(\alpha) - \frac{\alpha(\alpha-1)}{\sad(\sad-1)} \hat{G}_n(\sad)\right) - \ml(\alpha) \hat{G}'_n(\alpha),
\end{align}
where $\hat{G}_n(\alpha) := \int_0^{\alpha}\beta_n(s) \evt_0(s)^T\evt'_j(s) \dd s - \alpha \int_0^{1}\beta_n(s) \evt_0(s)^T\evt'_j(s) \dd s $. 
Using \eqref{eq:rem1-2} and the facts that $0 < \frac{\eta_0-\alpha}{\eta_0} < 1$ for $\alpha \in (0,\eta_0)$, $\ml(0)=0$ and $\| \hat{G}_n \|_{\Co([0,1];R)} \leq C \| \beta_n \|_{\Cz([0,1];\R)}$, we have
\begin{equation*}
    \lim_{n\rightarrow +\infty} \| \hat{G}_n \|_{\Co([0,1];\R)} = \lim_{n\rightarrow +\infty} \| \beta_n \|_{\Cz([0,1];\R)} = 0.
\end{equation*}
Together with \eqref{eq:rem1-3} and Lemma \ref{lemma:lambda} (iii), this implies $\lVert f_n \rVert_{\Y} \rightarrow 0$ as $n\rightarrow +\infty$. 
Meanwhile, we have 
\begin{equation*}
\lVert \tv_n \rVert_{\X} \geq \lVert \tv'_n(0) \rVert_{\ell^\infty} = \lVert \ml'(0)\evt_j(0) \rVert_{\ell^\infty}.
\end{equation*}
Therefore $\lVert f_n \rVert_{\Y}$ can not be uniformly bounded by $\lVert \tv_n \rVert_{\X}$ as $n$ increases, and the stability result does not hold.

Assume that there exists $1\leq j\leq N-1$ such that $\ev_j(0)<\sA=\ml'(0)$ (which is contradicted to the assumption that $\sA$ is the lowest eigenvalue).
Let 
\begin{eqnarray}
\label{rem1-3}
\sigma_j:=\frac{\ev_j(0)}{\ml'(0)} < 1.
\end{eqnarray}
Let $\eta_1:=\frac{\sad}{2}$ and
\begin{equation*}
\beta(\alpha) = \left\{
\begin{array}{ll}
\alpha^{\sigma_j}(\alpha - \eta_1)^2 &{\rm in} ~~[0, \eta_1]  \\[1ex]
0 &{\rm in} ~~(\eta_1,1]
\end{array} ,
\right.
\quad
\tv(\alpha) = \beta(\alpha)\evt_j(\alpha)
\quad{\rm and}\quad
f = \delta\mathscr{F}(\mv)\tv .
\end{equation*}
From \eqref{pb:beta}, \eqref{eq:beta0} and \eqref{eq:ev}, we obtain from a direct computation that
\begin{align}
\evt_j(\alpha)^T f(\alpha) &= \ev_j(\alpha)\beta(\alpha) - \ml(\alpha)\beta'(\alpha) = \ml(\alpha)\alpha^{\sigma_j}(\alpha-\eta_1)\left( \left( \frac{\ev_j(\alpha)}{\ml(\alpha)} - \frac{\sigma_j}{\alpha} \right)(\alpha-\eta_1)  - 2\right) ,
\nonumber\\[1ex]
\evt_i(\alpha)^T f(\alpha) &= -\ml(\alpha) \beta(\alpha) \evt_i(\alpha)^T\evt'_j(\alpha)\quad{\rm for}~ i = 1,\cdots,N-1,~i\neq j \qquad{\rm and}
\nonumber \\[1ex]
\evt_0(\alpha)^T f(\alpha) & = (\sA+\sB) \left( \hat{G}(\alpha) - \frac{\alpha(\alpha-1)}{\sad(\sad-1)} \hat{G}(\sad)\right) - \ml(\alpha) \hat{G}'(\alpha),\nonumber
\end{align}
where $\hat{G}(\alpha) := \int_0^{\alpha}\beta(s) \evt_0(s)^T\evt'_j(s) \dd s - \alpha \int_0^{1}\beta(s) \evt_0(s)^T\evt'_j(s) \dd s$. 
Then applying Lemma \ref{lemma:lambda} (iii) yields  $\|f\|_{\Y}<\infty$. However, using $\sigma_j<1$ leads to
\begin{equation*}
    \lim_{\alpha\rightarrow 0^+} \tv'(\alpha) = \lim_{\alpha\rightarrow 0^+}\beta'(\alpha) \evt_j(0) = \lim_{\alpha\rightarrow 0^+} \frac{\sigma_j(1-\eta_1)^2}{\alpha^{1-\sigma_j}}\evt_j(0) = \infty.
\end{equation*}
This implies that $f\in\Y$ has no preimage in $\X$. 
Therefore $\delta\mathscr{F}(\mv) :\X\rightarrow\Y$ is not a surjection, and thus not an isomorphism.


\section{Proofs: Applications of the stability}
\label{sec:app}
\setcounter{equation}{0}

\begin{proof}
[Proof of Corollary \ref{coro:app1}]
For sufficiently small $\delta$, there exists an interpolation $\Pi_{\delta}: \U\rightarrow \U_\delta$ such that $\Pi_\delta \mv\in C^1\big([0,1];\mathcal{U}_{\epsilon}\cap\HH_\delta\big)$, $\Pi_\delta\mv(0) = y^A_{M,\delta},~\Pi_\delta\mv(1)=y^B_{M,\delta}$ and
\begin{equation}
\label{eq:app1-3}
\|\Pi_\delta\mv - \mv\|_{C^1([0,1];\HH)} \leq C\inf_{\varphi_\delta\in\U_\delta}\|\varphi_\delta-\mv\|_{C^1([0,1];\HH)}
\end{equation}
where $C>1$.
To see the existence of such interpolation, we provide a simple construction by
\begin{eqnarray*}
\big(\Pi_\delta\mv\big)(\alpha) = I_{\delta}\big(\mv(\alpha)\big) + \alpha\big(y^B_{M,\delta}-I_{\delta}y^B_M\big) + (1-\alpha)\big(y^A_{M,\delta}-I_{\delta}y^A_M\big) 
\qquad\forall~\alpha\in[0,1],
\end{eqnarray*}
where $I_{\delta}:\HH\rightarrow\HH_{\delta}$ is given such that $|I_{\delta}x-x| = \inf_{x_{\delta}\in\HH_{\delta}}|x_{\delta}-x|$ for any $x\in\HH$.
Let $\dF:\U_\delta\rightarrow \Y$ be given by \eqref{def:F} with $E$ replaced by $\dE$. Using \eqref{eq:app1-3}, we have the consistency error
\begin{align}
\label{app1:cons}
\|\dF(\Pi_\delta \mv)\|_{\Y} &= \| \dF(\Pi_\delta\mv) - \mathscr F(\mv)\|_{\Y}
\nonumber\\[1ex]
&\leq C \|\nabla \dE(\Pi_\delta\mv) - \nabla E(\mv) \|_{C^1([0,1];\HH)} + C \|\Pi_\delta\mv - \mv\|_{C^1([0,1];\HH)} 
\nonumber \\[1ex]
& \leq C \|E_{\delta}-E\|_{C^1(\mathcal{U}_{\epsilon} \cap \HH_\delta)} + C\inf_{\varphi_\delta\in\U_\delta}\|\varphi_\delta-\mv\|_{C^1([0,1];\HH)},
\end{align}
where $C$ depends only on $E$ and $\mv$.
	
Define a subspace $\X_\delta$ of $\X$ by
\begin{equation*}
    \X_\delta := \{\psi\in C^1\big([0,1];\HH_\delta\big):~\psi(0)=\psi(1)=0\}.
\end{equation*}
To derive the stability of $\delta{\mathscr F}_{\delta}(\Pi_\delta\mv)$, using \eqref{eq:app1-3} and Theorem \ref{theorem:Stability}(i), we obtain by a direct calculation that
\begin{align*}
\|\delta{\mathscr F}(\mv)\tv - \delta\dF(\Pi_\delta \mv)\tv\|_{\Y} 
&\leq \|\delta{\mathscr F}(\mv)\tv - \delta{\mathscr F}(\Pi_\delta \mv)\tv\|_{\Y}  + \|\delta{\mathscr F}(\Pi_\delta \mv)\tv - \delta\dF(\Pi_\delta \mv)\tv\|_{\Y}  
\\[1ex]
& \hskip -2cm \leq C \left( \|E_{\delta}-E\|_{C^2(\mathcal{U}_{\epsilon} \cap \HH_\delta)} + \inf_{\varphi_\delta\in\U_\delta}\|\varphi_\delta-\mv\|_{C^1([0,1];\HH)}\right) \|\tv\|_{\X}
\qquad\forall~\tv\in\X_\delta ,
\end{align*}
where the constant $C$ depends only on $E$ and $\mv$.
Taking sufficiently small $\delta$ and using Theorem \ref{theorem:Stability}(ii), we have that
\begin{align}
\label{app1:stab}
\|\delta\dF(\Pi_\delta \mv)\tv\|_{\Y}& \geq \|\delta{\mathscr F}(\mv)\tv\|_{\Y} - \|\delta{\mathscr F}(\mv)\tv - \delta\dF(\Pi_\delta \mv)\tv\|_{\Y} \geq C\|\tv\|_{\X} \quad~\forall~\tv\in\X_\delta,
\end{align}
where $C$ depends only on $E$ and $\mv$.

Combining the consistency \eqref{app1:cons} and the stability \eqref{app1:stab}, we can apply the inverse function theorem \cite[Lemma 2.2]{2011qnl} to obtain that there exists a solution $\mvp\in\U_\delta$ of \eqref{mep-app1} and
\begin{equation*}
\|{\Pi}_\delta\mv-\mvp\|_{C^1([0,1];\HH)} \leq C \|E_{\delta}-E\|_{C^1(\mathcal{U}_{\epsilon} \cap \HH_\delta)} + C \inf_{\varphi_\delta\in\U_\delta}\|\varphi_\delta-\mv\|_{C^1([0,1];\HH)}.
\end{equation*}
This together with the triangle inequality yields the estimate \eqref{eq:Corollar1}.
\end{proof}

\section{Conclusions}
\label{sec:conclusion}
\setcounter{equation}{0}

This paper provides a stability of the MEP, showing that the perturbation of a curve from the MEP can be controlled by the corresponding force under appropriate norms. 
The stability result has many important consequences on both theoretical and numerical aspects for the MEP.
For example, we show that the MEP stays close to the original one within a small perturbation of the energy landscape or the discretization of the configuration space. 
Our results also provide a crucial foundation for a convergence analysis of the string and NEB methods \cite{dMEPconv}.

\small
\bibliographystyle{plain}
\bibliography{reference}

\end{document}